\documentclass[trsc]{informs3}

\OneAndAHalfSpacedXI  

\def\ABSfont{\normalsize}

\usepackage{mathtools}
\usepackage{bm}

\usepackage{booktabs}
\usepackage{tabularx}
\usepackage{float}
\usepackage{tikz}
\usetikzlibrary{decorations.pathreplacing,arrows.meta,positioning}
\usepackage{subcaption}

\usepackage{enumitem}
\usepackage{framed}
\usepackage{algorithm}
\floatstyle{ruled}
\restylefloat{algorithm}
\floatstyle{plain}

\usepackage[colorlinks,citecolor=blue,linkcolor=blue,urlcolor=blue]{hyperref}
\usepackage[capitalise,noabbrev]{cleveref}

\definecolor{ured}{RGB}{255, 75, 0}
\definecolor{ublue}{RGB}{0, 90, 255}

\EquationsNumberedThrough
\TheoremsNumberedBySection

\crefname{theorem}{Theorem}{Theorems}
\crefname{lemma}{Lemma}{Lemmas}
\crefname{proposition}{Proposition}{Propositions}
\crefname{corollary}{Corollary}{Corollaries}
\crefname{conjecture}{Conjecture}{Conjectures}
\crefname{definition}{Definition}{Definitions}
\crefname{assumption}{Assumption}{Assumptions}
\crefname{remark}{Remark}{Remarks}
\crefname{example}{Example}{Examples}
\crefname{equation}{Eq.}{Eqs.}
\Crefname{equation}{Eq.}{Eqs.}
\crefname{algorithm}{Algorithm}{Algorithms}
\Crefname{algorithm}{Algorithm}{Algorithms}

\usepackage{natbib}
\bibpunct[, ]{(}{)}{,}{a}{}{,}%
\def\bibfont{\small}%
\newcommand{\ClO}{\mathcal{O}}
\newcommand{\ClS}{\mathcal{S}}
\newcommand{\ClT}{\mathcal{T}}
\newcommand{\ClN}{\mathcal{N}}
\newcommand{\BbR}{\mathbb{R}}

\newcommand{\Vttheta}{\boldsymbol{\theta}}
\newcommand{\Vtdelta}{\boldsymbol{\delta}}
\newcommand{\slotmap}{\sigma}

\newenvironment{thm}[1][]{\begin{theorem}[#1]}{\end{theorem}}
\newenvironment{lem}[1][]{\begin{lemma}[#1]}{\end{lemma}}
\newenvironment{pro}[1][]{\begin{proposition}[#1]}{\end{proposition}}
\newenvironment{cor}[1][]{\begin{corollary}[#1]}{\end{corollary}}

\newenvironment{dfn}[1][]{\begin{definition}[#1]}{\end{definition}}
\newenvironment{asm}[1][]{\begin{assumption}[#1]}{\end{assumption}}

\newenvironment{prf}[1][Proof]{\proof{#1}}{\endproof}

\graphicspath{{fig/}}

\RUNAUTHOR{Sakai and Kawase}
\RUNTITLE{Coarse Reporting in the Bottleneck Model}
\renewcommand{\theMANUSCRIPTNO}{\,}

\TITLE{Coarse Preference Reporting in the
Bottleneck Model: Approximate Strategyproofness and Efficiency}

\ARTICLEAUTHORS{%
\AUTHOR{Takara Sakai}
\AFF{Institute of Science Tokyo, \EMAIL{sakai.t.dcad@m.isct.ac.jp}}
\AUTHOR{Riki Kawase}
\AFF{Institute of Science Tokyo, \EMAIL{kawase.r.8d18@m.isct.ac.jp}}
}

\ABSTRACT{%
A central operator schedules each vehicle's passage time through a
bottleneck to achieve a dynamic system optimum (DSO).
The assignment depends on each vehicle's preferred arrival time,
which is private and must be elicited from each vehicle.
Mechanisms that elicit exact preferences, such as the
Vickrey--Clarke--Groves (VCG) mechanism, can achieve
strategyproofness but involve relatively complex rules and a
computational burden on the operator.
We focus instead on coarse reporting, in which each vehicle selects
from a finite menu of time slots of a common width.
This discrete interface already structures reservation and
appointment systems in practice, including managed lanes for
automated vehicles, airport slot allocation, and delivery
appointment windows.
We design a slot-based DSO mechanism on this coarse interface, in
which the operator implements DSO assignment based on the reported
slots and charges a capacity shadow price as a toll, and evaluate
its performance.
We prove that both the worst-case misreporting gain and the expected
efficiency loss decrease quadratically in the slot width.
The efficiency loss decays in this way under binding capacity, while
the worst-case misreporting gain requires an additional condition on
the preferred arrival time distribution and the schedule cost
function.
Analyzing the no-toll case, we find that the misreporting incentive
persists, however finely the slots are refined, indicating that the
toll also serves to elicit truthful reports.
Numerical experiments support these theoretical results and show
that they continue to hold in parameter regions outside the
sufficient conditions.}


\begin{document}
\maketitle

\section{Introduction}
\label{sec:Introduction}

\subsection{Background and Motivation}

A central operator can directly control the time
at which each vehicle passes through a capacity bottleneck in
several transportation systems, including managed lanes
for automated vehicles, airport runway slots, and reservation
based access to congested facilities.
Each vehicle has a preferred arrival time, and arriving earlier or
later incurs a schedule cost.
The operator schedules the passages to minimize the total schedule
cost under the bottleneck capacity, which defines the dynamic
system optimal (DSO) assignment
\citep{Vickrey1969-ic,Arnott1990-pc}.
Each vehicle's preferred arrival time is private information.
Since the DSO assignment depends on the preference distribution,
the operator must elicit preferences truthfully.
The required property is strategyproofness.
No vehicle should benefit from misreporting its preferred arrival
time.

The Vickrey--Clarke--Groves (VCG) mechanism
\citep{Clarke1971-hh,Groves1973-cu} is one approach to this
problem.
VCG computes the socially optimal assignment and charges each
vehicle an individualized transfer equal to the externality it
imposes on other vehicles, thereby achieving exact
strategyproofness and efficiency.
\citet{WadaAkamatsu2013-hy} and \citet{Wang2019-tp} apply
Groves type payments to the trading of bottleneck passage
permits and establish dominant strategy truthfulness in that
setting.
The VCG mechanism, however, has practical
limitations.
Its transfer rule, based on the counterfactual social welfare
without each vehicle, is computationally costly and difficult for
participants to verify
\citep{AusubelMilgrom2006-vm,Rothkopf2007-vg}.
As an alternative that avoids these limitations, we study coarse
preference reporting, in which each vehicle selects from a finite
menu of time slots rather than reporting an exact arrival time.
This coarse reporting corresponds to the discrete interface of
reservation and appointment systems widely used in practice,
which is simpler to operate and to participate in.

We study this coarse preference reporting
interface and develop a theoretical basis for evaluating its
performance.
Instead of eliciting exact continuous preferences, the
interface restricts reports to a finite menu of
time slots.
The operator solves a slot-based DSO problem given the reported
slots and charges the capacity shadow price as a
toll.
We characterize how this already prevalent
coarse interface performs as a function of its reporting
resolution.
We answer two questions:
how fast do the worst-case misreporting incentive and the expected
efficiency loss vanish as the slot menu is refined, and is a
pricing instrument necessary to control manipulation, or does the
slot menu alone suffice?

\subsection{Main Results}

The coarse interface we analyze has each vehicle
select a time slot from a uniform partition of width $\Delta$.
The operator solves a slot-based DSO problem given the reported
slots and charges each vehicle a toll equal to the shadow price
of capacity at the assigned arrival time.
This problem has the same structure as the DSO assignment under
discrete heterogeneity of preferred arrival times
\citep{Lindsey2004-mu,Akamatsu2021-di}.
Because the operator centrally controls vehicle passage times,
the DSO flow pattern under the reported preferences is
implemented directly.
The toll is the capacity shadow price arising
from this assignment.

Our central contribution is the design of a
slot-based DSO mechanism on a coarse, finite resolution preference
interface for bottleneck scheduling and the evaluation of its
performance.
We quantify how its incentive and efficiency properties depend on
the single operational parameter $\Delta$.
Such a coarse interface also simplifies participation and
implementation relative to exact elicitation schemes such as VCG.

We summarize the main results.

First, we prove that the slot-based mechanism
is second order approximately strategyproof.
The worst-case misreporting gain
$\varepsilon^{\ast}(\Delta)$, i.e., the maximum cost reduction
achievable by deviating from truthful reporting, satisfies
$\varepsilon^{\ast}(\Delta) = \ClO(\Delta^{2})$ for every slot
width up to a smoothness-determined upper threshold, under a
peakedness and curvature condition on the preferred arrival time
distribution and the schedule cost function.
The quadratic rate arises because the toll absorbs the leading
effect of a one-slot deviation, leaving only a second order
residual.
Halving the slot width then quarters the worst-case misreporting
gain.

Second, we quantify the efficiency loss.
Because the operator determines the assignment from reported
slots rather than exact preferred arrival times, wider slots
leave more room for efficiency loss relative to the continuous
DSO benchmark.
We establish that the expected efficiency loss
$L^{\ast}(\Delta)$ satisfies
$L^{\ast}(\Delta) = \ClO(\Delta^{2})$ for every $\Delta > 0$
under binding capacity.

The misreporting gain and the efficiency loss
both decay as $\ClO(\Delta^{2})$, so refining the slot menu
improves strategyproofness and efficiency together over the
operational range.
In the nonatomic (price-taking) limit the efficiency guarantee
holds for every $\Delta > 0$, while the strategyproofness guarantee
holds for every $\Delta \in (0, \Delta_{\max}]$, where
$\Delta_{\max}$ keeps the higher-order smoothness remainder dominated
by the concavity margin and leaves enough slots for the worst
deviation (\cref{sec:strategyproofness} and \cref{app:proofs}).

The role of the toll emerges from comparing the
mechanism with its no-toll counterpart.
We prove that without the toll, the misreporting gain stays
bounded away from zero regardless of how finely the slot menu is
partitioned.
Because the operator already enforces the DSO assignment, the
toll is not needed to induce the system optimum. Instead, it
elicits truthful reporting under the coarse interface.
This contrasts with its role in the classical bottleneck model
\citep{Vickrey1969-ic,Arnott1990-pc,Arnott1993-lu}, where the
toll internalizes the queueing externality and shapes the
equilibrium departure pattern.
We develop this comparison in \cref{subsec:role_of_tolling}.

\subsection{Related Work}
\label{subsec:related_work}

The bottleneck model of \citet{Vickrey1969-ic}, extended by
\citet{Arnott1990-pc,Arnott1993-lu}, characterizes equilibrium departure patterns and the optimal
toll under a range of assumptions on user heterogeneity and
network structure.
\citet{Small2015-lc} surveys this literature and
\citet{Li2020-zi} provides a bibliometric review.
Efficient tolling schemes have been derived for heterogeneous
user classes
\citep{ArnottDePalmaLindsey1994-wc,Lindsey2004-mu,
vandenBergVerhoef2011-wd,Akamatsu2021-di}
and corridor or tandem bottleneck networks
\citep{Fu2022-st,Sakai2024-wy}.
These models assume that the operator observes preferences
exactly.
Tradable bottleneck permit schemes
\citep{WadaAkamatsu2013-hy,Wang2019-tp} introduce mechanism
design into this setting through Groves type auction markets,
but they require iterated bidding with precise valuations rather
than the one shot coarse reporting studied here.

In the mechanism design literature,
\citet{Azevedo2018-yk}, \citet{Budish2012-hg}, and
\citet{Balseiro2024-jn} show that coarsening the action space
can yield approximate incentive compatibility in matching and
allocation problems.
\citet{NisanSegal2006-je} and \citet{Blumrosen2007-jc}
characterize how efficiency degrades under bounded
communication.
None of these frameworks parameterizes the coarsening by an
operational quantity such as the slot width in a bottleneck
or derives a convergence rate in that parameter.

Discrete time slots serve as the operational interface in airport
operations \citep{Zografos2017-be}, reservation based demand
management for autonomous traffic \citep{Lamotte2017-ap}, and
dedicated freight lanes \citep{Razmi-Rad2020-nb}, but these
works fix the slot structure rather than treating its width as a
design variable.
A parallel logistics literature
\citep{Agatz2011-tsm,Strauss2021-dpf,Chen2013-mta,Phan2016-cts}
treats the slot width as a design variable in attended home
delivery and port terminal appointments, studying operational
consequences without analyzing strategic incentives.

To our knowledge, no prior work analyzes the strategyproofness
and efficiency of a coarse preference reporting mechanism for
bottleneck scheduling as functions of the reporting resolution.
The slot width is a natural design variable in this context:
as automated and connected vehicle systems expand, the
operator must choose a reporting granularity that balances
simplicity against performance.
The present analysis provides the first quantitative tool for
this choice, giving explicit convergence rates that link the slot
width to both the strategyproofness gap and the efficiency loss.
The theoretical contribution lies in connecting the bottleneck
model literature, which has not addressed strategic
reporting under coarse information, with the mechanism design
literature, which has not parameterized coarsening by an
operational quantity in a congestion model.

\subsection{Paper Organization}

\cref{sec:model} introduces the model framework, comprising the
bottleneck setting, the users, and the schedule delay cost.
\cref{sec:mechanism} presents the proposed slot-based DSO mechanism
and establishes its basic structural properties.
\cref{sec:strategyproofness,sec:efficiency} analyze the
strategyproofness and efficiency of the mechanism and establish the two
quadratic rates that constitute our main results.
\cref{sec:numerical} evaluates the mechanism numerically and
interprets the magnitudes at operationally relevant slot widths.
\cref{sec:discussion} discusses the role of the toll and the broader
operational implications.
\cref{sec:conclusion} concludes, with all proofs collected in the
appendix.

\section{Basic Assumptions}
\label{sec:model}

\cref{subsec:network,subsec:vehicles} state the network and user
assumptions.
The $\Delta$-DSO objects
($q^{\ast}_\Delta$, $p^{\ast}_\Delta$,
$\lambda^{\ast}_\Delta$, $\Gamma^{\ast}_\Delta$) have a
$\Delta$ subscript, and we drop it
($q^{\ast}, p^{\ast}, \lambda^{\ast}, \Gamma^{\ast}$) when
$\Delta$ is fixed.
\cref{tab:notation} in the appendix summarizes notation.

\subsection{Network}
\label{subsec:network}
Consider a single-origin, single-destination network in which an
origin and a destination are connected by an expressway with a
single bottleneck (\cref{fig:network}).
The bottleneck capacity, defined as the maximum number of
vehicles that can pass per unit time, is denoted by $\mu$.
Following the single-bottleneck model of \citet{Vickrey1969-ic},
the operator maintains the bottleneck inflow at or below $\mu$
so that no queue arises.
The free-flow travel time is denoted by~$f$. A vehicle departing
at time $t_{\mathrm{d}}$ arrives at time $t = t_{\mathrm{d}} + f$.
We formulate the model in the arrival time~$t$.

\begin{figure}[tbp]
  \centering
  \includegraphics[clip, width=0.55\textwidth]{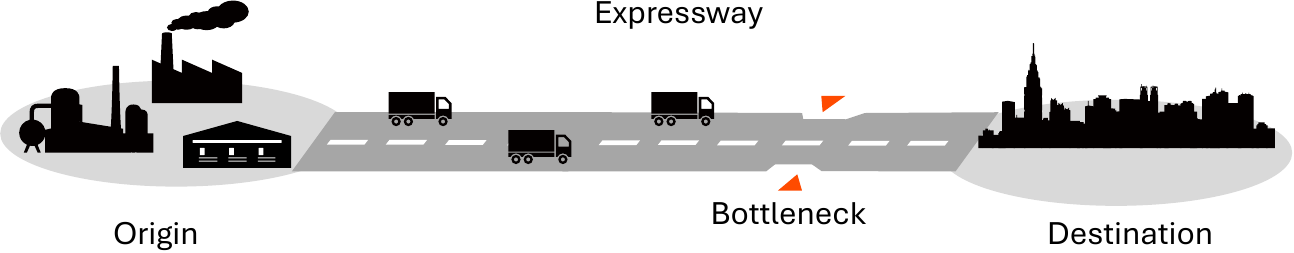}
  \caption{Single-origin, single-destination expressway with one
  bottleneck of capacity $\mu$.}
  \label{fig:network}
\end{figure}

\subsection{Users and Schedule Delay Cost}
\label{subsec:vehicles}
Let $\ClN = \{1, 2, \ldots, N\}$ denote the set of users (vehicles)
served by this network.
The operator schedules arrivals within an
\emph{assignment window} $\ClT = [0, T] \subset \BbR$ of
length~$T$.
Each user $n \in \ClN$ has a preferred arrival time
$\theta_{n} \in \ClS$ at the destination, which is private
information.
Here
$\ClS = [\theta^{-}, \theta^{+}] \subseteq \ClT$
is the support of preferred arrival times, with
width $S \coloneqq \theta^{+} - \theta^{-}$.
The profile of preferred arrival times across all vehicles is
denoted by
$\Vttheta = (\theta_{1}, \ldots, \theta_{N}) \in \ClS^{N}$.
Vehicles are homogeneous, with heterogeneity entering only through the
preferred arrival times $\theta_{n}$.

We adopt the nonatomic (fluid) limit. The population is
represented by the arrival rate density $\nu$ of
\cref{asm:demand_density}, each user has negligible mass, and a
single user therefore cannot alter the aggregate assignment or toll.
The strategyproofness analysis of \cref{sec:strategyproofness} is
accordingly conducted in this price-taking limit (strategyproofness
in the large in the sense of \citet{Azevedo2018-yk}), consistent
with the deterministic representative assumption~(D3) of
\cref{asm:demand_density}.
For a finite population the induced manipulation is $\ClO(1/N)$
(the price impact of one agent among $N$, vanishing in the
large market limit of \citet{Azevedo2018-yk}), and the numerical
experiments of
\cref{sec:numerical} use the finite population mechanism and confirm
that this correction is negligible at the tested sizes
($d(s) = 10$--$60$ vehicles per slot).
Exact dominant-strategy incentive compatibility for a fixed finite
population would instead require the personalized payments of a
Vickrey--Clarke--Groves mechanism, which the coarse anonymous toll
is designed to avoid.
The index set $\ClN = \{1, \ldots, N\}$ and the type $\theta_{n}$ then
serve as labels for a representative user, and $N$ is the total mass
$\int_{\ClS} \nu$.
\begin{asm}[Properties of the preferred arrival time distribution]
  \label{asm:demand_density}
  The preferred arrival times $\{\theta_{n}\}_{n \in \ClN}$ are
  distributed with an \emph{arrival rate density} $\nu$ (units:
  vehicles per unit time) that satisfies the following:
  \begin{enumerate}
    \item[(D1)] (Bounded support and total mass.) $\nu$ is
    supported on the closed interval
    $\ClS = [\theta^{-}, \theta^{+}] \subseteq \ClT$, i.e.,
    $\nu(\theta) \geq 0$ on $\ClS$, $\nu(\theta) = 0$ outside
    $\ClS$, and $\int_{\ClS} \nu(\theta)\,\mathrm{d}\theta = N$.
    \item[(D2)] (Lipschitz continuity.) There exists a constant
    $L_{\nu} \geq 0$ such that
    \begin{align}
      |\nu(\theta) - \nu(\theta')|
      \leq L_{\nu}\,|\theta - \theta'|,
      \qquad \forall \theta, \theta' \in \ClS.
      \label{eq:nu_Lipschitz}
    \end{align}
    \item[(D3)] (Deterministic representative.) The realized
    slotwise count is treated as its expectation under $\nu$,
    eliminating sampling noise in the assignment problem.
  \end{enumerate}
\end{asm}
We write
$\nu^{\mathrm{peak}} \coloneqq \max_{\theta \in \ClS} \nu(\theta)$
for the peak arrival rate of the prior.%
\label{eq:nu_peak_def}

Let $c(\theta, t)$ denote the schedule delay cost incurred when a
vehicle whose preferred arrival time is $\theta$ actually arrives
at time $t$.
The function $c(\theta, t)$ is called the
\emph{schedule delay cost function}.
We write $\dot{c}$ and $\ddot{c}$ for the first and second
derivatives of $c$ with respect to $\theta$ (equivalently,
with respect to~$t$ up to sign by~(C1) below). The subscripted
extrema $\dot{c}_{\max}, \ddot{c}_{\min}, \ddot{c}_{\max}$ refer
to the corresponding bounds taken in absolute value over the
domain $(\theta, t) \in \ClS \times \ClT$.
It satisfies the following properties:
\begin{asm}[Properties of the schedule delay cost function]
  \label{asm:schedule_cost}
  The schedule delay cost function $c(\theta, t)$ satisfies the
  following properties:
  \begin{enumerate}
    \item[(C1)] $c(\theta, t)$ depends only on the deviation between
    the preferred and actual arrival times $(\theta - t)$:
    that is, $\theta - t = \theta' - t'$ implies
    $c(\theta, t) = c(\theta', t')$.
    $c(\theta, t)$ is continuously differentiable in both $\theta$
    and $t$, and satisfies
    $c(\theta, \theta) = \min_{t} c(\theta, t) = 0$.
    \item[(C2)] $c(\theta, t)$ is convex in $(\theta - t)$.
    \item[(C3)] $c(\theta, t)$ is Lipschitz continuous. There exists
    a constant $\dot{c}_{\max} \geq 0$ such that
    \begin{align}
      \bigl|c(\theta, t) - c(\theta', t)\bigr| \leq \dot{c}_{\max} \left|\theta - \theta'\right|,
      \quad \forall \theta, \theta' \in \ClS,\ \forall t \in \ClT
      \label{eq:Lip_schedule}
    \end{align}
    holds. By (C1), $|c(\theta, t) - c(\theta, t')| \leq \dot{c}_{\max}|t - t'|$
    also holds with the same constant.
    \item[(C4)] (Piecewise $C^{2}$ regularity.)
    There exists a finite partition of the deviation domain
    $z = \theta - t$ into intervals such that, on each interval,
    $c$ is twice continuously differentiable in~$z$ and satisfies
    \begin{align}
      0 < \ddot{c}_{\min}
      \leq \frac{\partial^{2} c}{\partial \theta^{2}}(\theta, t)
      \leq \ddot{c}_{\max}
      \label{eq:C2_regularity}
    \end{align}
    for all $(\theta, t) \in \ClS \times \ClT$ with $z = \theta - t$
    off the partition boundaries.
    The lower bound $\ddot{c}_{\min} > 0$ ensures uniform strict
    convexity.
    For an asymmetric quadratic cost
    $c(\theta, t) = \beta(\theta - t)^{2}\mathbf{1}_{t \le \theta}
    + \gamma(t - \theta)^{2}\mathbf{1}_{t > \theta}$ (used in
    \cref{sec:numerical}),
    $\ddot{c}_{\min} = 2\min(\beta, \gamma)$ and
    $\ddot{c}_{\max} = 2\max(\beta, \gamma)$.
    The strict convexity in (C4) ($\ddot{c}_{\min} > 0$) is
    invoked for the $\mathcal{O}(\Delta^{2})$ bounds in
    \cref{sec:strategyproofness,sec:efficiency}.
  \end{enumerate}
\end{asm}
Conditions (C1)--(C2) imply submodularity. If $\theta < \theta'$ and
$t < t'$, then
\begin{align}
  c(\theta, t) + c(\theta', t') \leq c(\theta, t') + c(\theta', t)
  \label{eq:submodular}
\end{align}
holds, which follows from convexity in $(\theta - t)$.
This submodularity condition means that assigning earlier
arrival times to vehicles with earlier preferred arrival times reduces
total cost.

\section{The Slot-Based DSO Mechanism}
\label{sec:mechanism}

This section develops the slot-based DSO mechanism.
\cref{subsec:mechanism} states the operational rules,
\cref{subsec:dso_delta} the operator's response to a reported
profile, \cref{subsec:FIFW} the structural identities of the
optimal assignment and toll, and \cref{subsec:individual_cost}
the per-vehicle cost used in
\cref{sec:strategyproofness,sec:efficiency}.

\subsection{Operational Rules}
\label{subsec:mechanism}
The managed lane for automated vehicles,
including its bottleneck section, is controlled
by an operator.
Each user $n$ reports its preferred arrival time to the operator,
who then determines the assignment and tolls based on these reports.
To use the facility, a user must pay the operator a toll.

The information structure of the mechanism
is as follows.
Each vehicle's true preferred arrival time $\theta_{n}$ is
\textbf{private information}.
The operator observes only the self-reported values
$\hat{\theta}_{n} \in \ClS_\Delta$.
The assignment and toll setting are determined solely on the basis of
the report profile $\hat{\Vttheta}$.
Because each user is a price taker (\cref{subsec:vehicles}), a deviating
user takes the aggregate assignment and toll as given. It need not
observe the reports or true preferred arrival times of other users.
In the proposed mechanism, the set of admissible preferred arrival
time reports is restricted to a finite number of equally spaced
time slots.
The slot width $\Delta > 0$ determines the reporting resolution.
We call this restriction $\Delta$-coarsening:
\begin{dfn}[$\Delta$-coarsening]
  \label{dfn:coarsening}
  $\Delta$-coarsening is the operation of restricting preferred
  arrival time reports to the equally spaced finite set
  $\ClS_{\Delta} = \{s_{1}, s_{2}, \ldots, s_{M}\}$.
  Here, $\Delta > 0$ is called the coarsening parameter (slot width)
  and satisfies $s_{i+1} - s_{i} = \Delta$.
  Each slot $s_{i}$ represents the midpoint (representative time) of
  the time interval $[s_{i} - \Delta/2,\ s_{i} + \Delta/2)$ contained
  in that slot.
\end{dfn}
The $M$~slots cover the assignment window $\ClT = [0, T]$.
The slot width $\Delta$ is chosen as a divisor of~$T$, so that
$M = T/\Delta$ is a positive integer.
The slot nearest to vehicle $n$'s
preferred arrival time $\theta_{n}$ is denoted by
\begin{align}
  r_{\Delta}(\theta_{n}) \coloneqq \arg\min_{s \in \ClS_{\Delta}} \left|s - \theta_{n}\right|
  \label{eq:rounding}
\end{align}
(with ties broken by an arbitrary convention).
The mapping $r_\Delta$ is called the rounding map.
The coarsening error is bounded by
$|\theta_{n} - r_\Delta(\theta_{n})| \leq \Delta/2$.
The slot width $\Delta$ controls the reporting resolution. In the
limit $\Delta \to 0$, the set of admissible slots approaches a
continuous time axis.

Given the slot width $\Delta$, the mechanism proceeds as in
\cref{alg:dso}.

\begin{algorithm}[t]
  \caption{Slot-based DSO mechanism (slot width~$\Delta$)}
  \label{alg:dso}
  \begin{enumerate}[leftmargin=2.2em, itemsep=2pt, topsep=4pt,
                    label=\textup{\arabic*.}]
    \item The operator announces $\Delta$ and the slot set
      $\ClS_{\Delta}$.
    \item Each user $n$ reports $\hat{\theta}_{n} \in \ClS_{\Delta}$
      (possibly strategically, not necessarily
      $r_{\Delta}(\theta_{n})$), giving the profile
      $\hat{\Vttheta} = (\hat{\theta}_{1}, \ldots, \hat{\theta}_{N})$.
    \item The operator solves the slot-based DSO problem
      (\cref{subsec:dso_delta}) for the inflow pattern
      $q^{\ast}_{\Delta}(s,t)$ and toll $p^{\ast}_{\Delta}(t)$.
    \item Within each slot $s$, the $d(s)$ users are placed
      uniformly at random over the $d(s)$ equispaced positions in
      $\Gamma^{\ast}_\Delta(s)$, independently across slots.
    \item Users pass at the assigned times and pay the corresponding
      tolls.
  \end{enumerate}
\end{algorithm}

\cref{fig:framework} illustrates the mechanism.
Each vehicle has a true preferred arrival time~$\theta_{n}$
(left axis), which it maps to a reported slot~$\hat{\theta}_{n}$
(center axis) via the rounding map~$r_\Delta$.
A vehicle may report truthfully
($\hat{\theta}_{n} = r_\Delta(\theta_{n})$, as user~$n$ in the
figure) or misreport to a different slot
($\hat{\theta}_{n} \neq r_\Delta(\theta_{n})$, as user~$n'$).
The operator observes only the reported slots and solves the
slot-based DSO problem to determine the assignment interval
$\Gamma^{\ast}_\Delta(s)$ for each slot~$s$ (right axis).

\begin{figure}[tbp]
  \centering
  \begin{tikzpicture}[
    >=stealth,
    slot/.style={fill=black!8, draw=none},
    intv/.style={fill=blue!12, draw=blue!40, semithick},
    lbl/.style={font=\footnotesize},
    every node/.style={font=\small},
  ]
    \def\xT{0}        
    \def\xR{4.5}      
    \def\xA{8.8}      

    \def\sA{0.8}   \def\sB{2.4}   \def\sC{4.0}
    \def\hw{0.8}

    \draw[rounded corners=5pt, green!60!black, thick]
      (3.2,-0.8) rectangle (10.4,6.2);
    \node[green!60!black, font=\small\bfseries,
      fill=white, inner sep=2pt]
      at (6.8,6.2) {Operator: $\Delta$-DSO assignment};

    \node[blue, font=\small\bfseries, align=center]
      at (\xT,5.3) {True preferred\\[-1pt]arrival time};
    \node[red!80!black, font=\small\bfseries, align=center]
      at (\xR,5.2) {Reported\\[-1pt]slot};
    \node[font=\small\bfseries, align=center]
      at (\xA,5.2) {Assigned\\[-1pt]arrival time};

    \foreach \x in {\xT,\xR,\xA} {
      \draw[gray!50, thin] (\x,-0.5) -- (\x,4.6);
    }

    \foreach \sc/\lab in {\sA/$s_{1}$, \sB/$s_{2}$, \sC/$s_{3}$} {
      \fill[slot] (\xR-0.25, \sc-\hw) rectangle (\xR+0.25, \sc+\hw);
      \draw[gray!50, thin]
        (\xR-0.25, \sc+\hw) -- (\xR+0.25, \sc+\hw);
      \draw[gray!50, thin]
        (\xR-0.25, \sc-\hw) -- (\xR+0.25, \sc-\hw);
      \draw[thick] (\xR-0.1, \sc) -- (\xR+0.1, \sc);
      \node[right, lbl] at (\xR+0.3, \sc) {\lab};
    }

    \draw[<->, thin] (\xR+0.75, \sB-\hw) -- (\xR+0.75, \sB+\hw);
    \node[right, lbl] at (\xR+0.8, \sB) {$\Delta$};

    \foreach \sc in {\sA, \sB, \sC} {
      \fill[intv] (\xA-0.15, \sc-\hw) rectangle (\xA+0.15, \sc+\hw);
    }
    \foreach \bd in {0.0, 1.6, 3.2, 4.8} {
      \draw[blue!40, semithick] (\xA-0.15, \bd) -- (\xA+0.15, \bd);
    }

    \draw[<->, thin] (\xA+0.4, \sB-\hw) -- (\xA+0.4, \sB+\hw);
    \node[right, lbl] at (\xA+0.45, \sB) {$\Gamma^*_\Delta(s_{2})$};

    \foreach \sc in {\sA, \sB, \sC} {
      \draw[->, gray!40, thin, dashed]
        (\xR+0.25, \sc) -- (\xA-0.14, \sc);
    }

    \def\thn{2.6}
    \fill[blue] (\xT, \thn) circle (3pt);
    \node[above left, blue, lbl] at (\xT-0.08, \thn) {$\theta_{n}$};
    \node[left, lbl] at (-0.8, \thn) {User $n$};

    \draw[->, blue, thick]
      (\xT+0.12, \thn) -- (\xR-0.25, \sB);

    \node[blue, font=\scriptsize\itshape]
      at (1.6, 3.5) {truthful reporting};
    \node[blue, font=\scriptsize,
      draw=blue!60, fill=white, rounded corners=2pt,
      inner sep=2pt]
      at (1.6, 3.1) {$\hat{\theta}_{n} = r_\Delta(\theta_{n})$};

    \def\thnp{1.0}
    \fill[red!80!black] (\xT, \thnp) +(-3pt,-3pt) rectangle +(3pt,3pt);
    \node[below left, red!80!black, lbl] at (\xT-0.08, \thnp)
      {$\theta_{n'}$};
    \node[left, lbl] at (-0.8, \thnp) {User $n'$};

    \draw[->, red!80!black, thick]
      (\xT+0.12, \thnp) to[out=15, in=210] (\xR-0.25, \sB);

    \node[red!80!black, font=\scriptsize,
      draw=red!60, fill=white, rounded corners=2pt,
      inner sep=2pt]
      at (1.6, 1.1) {$\hat{\theta}_{n'} \neq r_\Delta(\theta_{n'})$};
    \node[red!80!black, font=\scriptsize\itshape]
      at (1.6, 0.7) {misreporting};
  \end{tikzpicture}
  \caption{Slot-based reporting mechanism. Each vehicle's true
  preferred arrival time $\theta_{n}$ (left axis) is rounded to a
  reported slot $\hat{\theta}_{n}$ (center axis) via
  $r_\Delta$. The operator solves the slot-based DSO problem and
  returns the assignment interval $\Gamma^{\ast}_\Delta(s)$
  (right axis).}
  \label{fig:framework}
\end{figure}
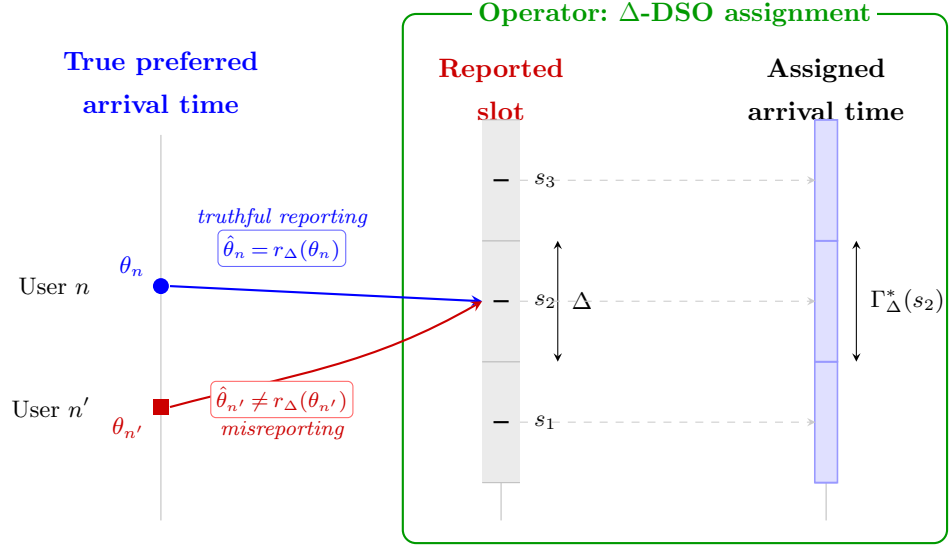

\subsection{The Slot-Based DSO Problem}
\label{subsec:dso_delta}
We refer to the operator's cost-minimization problem under
$\Delta$-coarsened reports (\cref{dfn:coarsening}) as the
\emph{slot-based DSO problem}, a discretized variant of the
classical DSO bottleneck problem
\citep{Arnott1990-pc,Arnott1993-lu}.
The per-slot count $d(s_{i})$ is determined by the report
profile (\cref{eq:demand} below) and is generally non-uniform
across slots.
From the report profile $\hat{\Vttheta}$, the total number of
vehicles reporting each slot (the aggregate demand) is
\begin{align}
  d(s) \equiv \sum_{n \in \ClN} \mathbf{1}_{\hat{\theta}_{n}= s},
  \qquad \forall s \in \ClS_{\Delta}
  \label{eq:demand}
\end{align}
where $\mathbf{1}_{A}$ is the indicator function of event $A$.
Under \cref{asm:demand_density}, we use the deterministic
representative $d(s_{i}) = \int_{[s_{i} - \Delta/2,\,
s_{i} + \Delta/2)} \nu(\theta)\,\mathrm{d}\theta$ for truthful
reports, so that
$|d(s_{i+1}) - d(s_{i})| \leq L_\nu \Delta^2$ for adjacent
slots within $\ClS$ and $d(s_{i}) = 0$ outside $\ClS$.
Using this, we formulate the operator's total schedule delay cost minimization
problem as follows:
\begin{align}
  \texttt{[$\Delta$-DSO]} \qquad
  \min_{\{q_{\Delta}(s,t)\}_{s, t}} \quad &
  \underbrace{\sum_{s \in \ClS_{\Delta}} \int_{t \in \ClT} c(s, t)\, q_{\Delta}(s, t) \,\mathrm{d}t}_{\text{total schedule delay cost under reported slots}}
  \label{eq:DSO_obj}
  \\
  \text{s.t.} \quad
  &\underbrace{\sum_{s \in \ClS_{\Delta}} q_{\Delta}(s, t)}_{\text{aggregate inflow at }t} \leq \mu
  &&\forall t \in \ClT,
  \label{eq:DSO_cap}
  \\
  &\underbrace{\int_{t \in \ClT} q_{\Delta}(s, t) \,\mathrm{d}t}_{\text{total flow assigned to slot }s} = d(s)
  &&\forall s \in \ClS_{\Delta},
  \label{eq:DSO_dem}
  \\
  &q_{\Delta}(s, t) \geq 0
  &&\forall s \in \ClS_{\Delta},\ \forall t \in \ClT.
  \label{eq:DSO_nonneg}
\end{align}
Here, $q_{\Delta}(s, t)$ is the flow rate assigned to actual arrival
time $t$ among vehicles that reported slot $s$.
Constraint \labelcref{eq:DSO_cap} is the capacity constraint, and
\cref{eq:DSO_dem} is the demand conservation condition.

Throughout this paper, we assume that the total demand~$N$ is
sufficiently concentrated relative to the bottleneck
capacity~$\mu$ so that the capacity constraint~\labelcref{eq:DSO_cap}
binds at all times during the demand period (the binding capacity
condition).
We choose the assignment window to coincide with the active period
and impose $T = N/\mu$, so that capacity is saturated throughout
the assignment window, as in the baseline of \cref{sec:numerical}.
Partial congestion, in which capacity is slack off-peak, lies
outside the present scope.
This condition ensures that the sorting property of the optimal
flow (\cref{subsec:FIFW}) and the toll structure
(\cref{subsec:toll_property}) are analytically tractable.

Note that the operator solves the slot-based DSO problem treating
the reported preferred arrival time $s \in \ClS_\Delta$ as the
preferred arrival time.
That is, the objective function \labelcref{eq:DSO_obj} uses the reported
value $s$ as the first argument of the schedule delay cost.
When a vehicle's true preferred arrival time $\theta_{n}$ differs from
its report $\hat{\theta}_{n}$, a discrepancy arises between the cost
perceived by the operator and the true cost incurred by the vehicle.
In \cref{sec:efficiency}, we analyze the impact of this discrepancy.

As in the classical bottleneck model
\citep{Arnott1990-pc,Small2015-lc}, the optimal toll that supports
the DSO assignment is obtained from the Lagrange multipliers of
the slot-based DSO problem.
Let $p_{\Delta}^{\ast}(t) \geq 0$ denote the multiplier on the
capacity constraint~\labelcref{eq:DSO_cap}, and let
$\lambda_{\Delta}^{\ast}(s)$ denote the multiplier on the demand
conservation condition~\labelcref{eq:DSO_dem}, which represents the
equilibrium generalized cost for slot~$s$.
We call $p^{\ast}_{\Delta}$ the
\emph{time-dependent managed lane toll},
abbreviated to the \emph{managed lane toll}. It
is the
Lagrange multiplier of the capacity constraint and therefore
coincides with the marginal externality of capacity usage at
time~$t$ under the coarsened reports.
If reports were continuous (i.e., $\Delta \to 0$), the same
determination reduces to the classical optimal congestion toll of
\citet{Vickrey1969-ic}. The managed lane toll
uses the same
formula but evaluates it under reports of resolution~$\Delta$.
The Karush--Kuhn--Tucker (KKT) optimality conditions are:
\begin{align}
  &\begin{dcases}
    c(s, t) + p^{\ast}_{\Delta}(t) = \lambda^{\ast}_{\Delta}(s) &\text{if} \quad q^{\ast}_{\Delta}(s, t) > 0
    \\
    c(s, t) + p^{\ast}_{\Delta}(t) \geq \lambda^{\ast}_{\Delta}(s) &\text{if} \quad q^{\ast}_{\Delta}(s, t) = 0
  \end{dcases}
  &&\forall s \in \ClS_{\Delta},\ \forall t \in \ClT,
  \label{eq:KKT_flow}
  \\
  &\begin{dcases}
    \sum_{s \in \ClS_{\Delta}} q^{\ast}_{\Delta}(s, t) = \mu &\text{if} \quad p^{\ast}_{\Delta}(t) > 0
    \\
    \sum_{s \in \ClS_{\Delta}} q^{\ast}_{\Delta}(s, t) \leq \mu &\text{if} \quad p^{\ast}_{\Delta}(t) = 0
  \end{dcases}
  &&\forall t \in \ClT.
  \label{eq:KKT_cap}
\end{align}
Condition \labelcref{eq:KKT_flow} states that for pairs $(s, t)$ with
positive flow, the generalized cost (schedule delay cost plus toll) equals
the equilibrium cost $\lambda^{\ast}_{\Delta}(s)$.
For pairs with zero flow, the generalized cost is at least as large
as the equilibrium cost.
This means that $p^{\ast}_{\Delta}(t)$ is a toll that internalizes the
congestion externality at time $t$.
Condition \labelcref{eq:KKT_cap} is a complementary slackness condition
stating that capacity is fully utilized at all times when the toll is
positive.

From \cref{eq:KKT_flow}, for any time $t$ at which
$p^{\ast}_{\Delta}(t) > 0$ (i.e., within the congested period),
the toll can be expressed as
\begin{align}
  p^{\ast}_{\Delta}(t) =
  \underbrace{\lambda^{\ast}_{\Delta}(\slotmap(t))}_{\substack{\text{equilibrium}\\\text{generalized cost}\\\text{of slot } \slotmap(t)}}
  -
  \underbrace{c(\slotmap(t), t)}_{\substack{\text{schedule delay cost}\\\text{at }t\text{ for a vehicle}\\\text{reporting }\slotmap(t)}},
  \label{eq:toll_expression}
\end{align}
where $\slotmap(t)$ denotes the slot with positive flow at time $t$
(under the sorting property, $\slotmap(t)$ is uniquely
determined, as shown in \cref{subsec:FIFW}).
Outside the congested period, $p^{\ast}_{\Delta}(t) = 0$.

\subsection{Assignment and Toll Pattern}
\label{subsec:FIFW}

We focus on the congested case $\nu^{\mathrm{peak}} > \mu$.
The trivial case $\nu^{\mathrm{peak}} \leq \mu$ gives
$p^{\ast}_{\Delta} \equiv 0$.
Under the submodularity \labelcref{eq:submodular} from (C1)--(C2),
the optimal $\Delta$-DSO assigns earlier reported slots to
earlier arrival times.

\begin{lem}[Sorting property of the $\Delta$-DSO flow pattern]
  \label{lem:FIFW}
  Under conditions (C1)--(C2) of \cref{asm:schedule_cost}, the optimal solution to the
  slot-based DSO problem assigns vehicles in the order of their
  reported preferred arrival times:
  if $s_{i} < s_{j}$, then for any $t, t'$ satisfying
  $q^{\ast}_{\Delta}(s_{i}, t) > 0$ and $q^{\ast}_{\Delta}(s_{j}, t') > 0$,
  we have $t \leq t'$.
\end{lem}
\begin{prf}
See \cref{app:proof_FIFW}.
\end{prf}

Define
$\Gamma^{\ast}_{\Delta}(s)
\coloneqq \{t \in \ClT \mid q^{\ast}_{\Delta}(s, t) > 0\}$.
Under \cref{lem:FIFW}, $\Gamma^{\ast}_{\Delta}(s)$ is a
contiguous interval
$[t^{-}_{\Delta}(s), t^{+}_{\Delta}(s)]$ with
$t^{+}_{\Delta}(s_{i}) = t^{-}_{\Delta}(s_{i+1})$, and where the
capacity binds
\begin{align}
  q^{\ast}_{\Delta}(s, t) =
  \mu \mathbf{1}_{t \in \Gamma^{\ast}_{\Delta}(s)},
  \qquad
  t^{+}_{\Delta}(s) - t^{-}_{\Delta}(s) = d(s)/\mu.
  \label{eq:flow_analytical}
\end{align}

\label{subsec:toll_property}%
The KKT condition~\labelcref{eq:KKT_flow} together with the
active period boundary
$p^{\ast}_{\Delta}(t^{-}_{\Delta}(s_{1})) = 0$ pins the
equilibrium cost recursively:
\begin{align}
  \lambda^{\ast}_{\Delta}(s_{1}) &= c(s_{1},\, t^{-}_{\Delta}(s_{1})),
  \label{eq:lambda_init}
  \\
  \lambda^{\ast}_{\Delta}(s_{i+1}) &=
  \lambda^{\ast}_{\Delta}(s_{i}) - c(s_{i},\, t^{+}_{\Delta}(s_{i}))
  + c(s_{i+1},\, t^{-}_{\Delta}(s_{i+1})),
  \label{eq:lambda_recursion}
\end{align}
where contiguity of adjacent assignment intervals
($t^{+}_{\Delta}(s_{i}) = t^{-}_{\Delta}(s_{i+1})$) follows from
\cref{lem:FIFW}.
The toll then reads
$p^{\ast}_{\Delta}(t) = \lambda^{\ast}_{\Delta}(s) - c(s, t)$ on
$\Gamma^{\ast}_{\Delta}(s)$ via~\cref{eq:toll_expression}, and the
Lipschitz continuity of $c$ transfers to the toll:
\begin{lem}[Lipschitz Continuity of the Toll]
  \label{lem:toll_Lip}
  Under \cref{asm:schedule_cost} and \cref{lem:FIFW}, the
  optimal toll $p^{\ast}_{\Delta}(t)$ satisfies
  \begin{align}
    \bigl|p^{\ast}_{\Delta}(t) - p^{\ast}_{\Delta}(t')\bigr|
    \leq L_{p} |t - t'|,
    \quad \forall t, t' \in \ClT,
    \label{eq:toll_Lip}
  \end{align}
  with $L_{p} = \dot{c}_{\max}$.
\end{lem}
\begin{prf}
See \cref{app:proof_toll_Lip}.
\end{prf}

\cref{fig:toll_pattern} illustrates the structure. Each
$\Gamma^{\ast}_{\Delta}(s_{i})$ is contiguous and ordered by
report (\cref{lem:FIFW}), and on each interval
$c(s_{i}, t) + p^{\ast}_{\Delta}(t) = \lambda^{\ast}_{\Delta}(s_{i})$
by the KKT condition~\labelcref{eq:KKT_flow} of the slot-based DSO
problem.
The toll is continuous across slot boundaries and vanishes at the
active-period endpoints.

\begin{figure}[tbp]
  \centering
  \begin{tikzpicture}[
    >=stealth,
    every node/.style={font=\small},
    eqC/.style={black!60, dashed, semithick},
    toll/.style={red!80!black, very thick},
  ]
    \draw[->, thick] (-0.2, 0) -- (12.2, 0)
      node[right, font=\small] {$t$};
    \draw[->, thick] (0, -0.2) -- (0, 4.4)
      node[above, font=\small] {cost};

    \def\xa{0.8}     
    \def\xz{10.8}    
    \pgfmathsetmacro{\xb}{0.8 + 1.25}
    \pgfmathsetmacro{\xc}{0.8 + 2.5}
    \pgfmathsetmacro{\xd}{0.8 + 3.75}
    \pgfmathsetmacro{\xe}{0.8 + 5.0}
    \pgfmathsetmacro{\xf}{0.8 + 6.25}
    \pgfmathsetmacro{\xg}{0.8 + 7.5}
    \pgfmathsetmacro{\xh}{0.8 + 8.75}

    \def\lA{1.6}
    \def\lB{2.4}
    \def\lC{2.9}
    \def\lD{3.1}
    \def\lE{3.1}
    \def\lF{2.9}
    \def\lG{2.4}
    \def\lH{1.6}

    \foreach \x in {\xa, \xb, \xc, \xd, \xe, \xf, \xg, \xh, \xz} {
      \draw[gray!70, dashed, semithick,
            dash pattern={on 3pt off 2pt}] (\x, 0) -- (\x, 3.7);
    }

    \draw[toll, smooth, samples=40, domain=\xa:\xb]
      plot (\x, {0.0 + 0.7 * (\x - \xa) / 1.25
                 + 0.45 * (\x - \xa) * (\xb - \x) / 1.5625});
    \draw[toll, smooth, samples=40, domain=\xb:\xc]
      plot (\x, {0.7 + 0.65 * (\x - \xb) / 1.25
                 + 0.45 * (\x - \xb) * (\xc - \x) / 1.5625});
    \draw[toll, smooth, samples=40, domain=\xc:\xd]
      plot (\x, {1.35 + 0.5 * (\x - \xc) / 1.25
                 + 0.45 * (\x - \xc) * (\xd - \x) / 1.5625});
    \draw[toll, smooth, samples=40, domain=\xd:\xe]
      plot (\x, {1.85 + 0.2 * (\x - \xd) / 1.25
                 + 0.45 * (\x - \xd) * (\xe - \x) / 1.5625});
    \draw[toll, smooth, samples=40, domain=\xe:\xf]
      plot (\x, {2.05 - 0.2 * (\x - \xe) / 1.25
                 + 0.45 * (\x - \xe) * (\xf - \x) / 1.5625});
    \draw[toll, smooth, samples=40, domain=\xf:\xg]
      plot (\x, {1.85 - 0.5 * (\x - \xf) / 1.25
                 + 0.45 * (\x - \xf) * (\xg - \x) / 1.5625});
    \draw[toll, smooth, samples=40, domain=\xg:\xh]
      plot (\x, {1.35 - 0.65 * (\x - \xg) / 1.25
                 + 0.45 * (\x - \xg) * (\xh - \x) / 1.5625});
    \draw[toll, smooth, samples=40, domain=\xh:\xz]
      plot (\x, {0.7 - 0.7 * (\x - \xh) / 1.25
                 + 0.45 * (\x - \xh) * (\xz - \x) / 1.5625});
    \node[red!80!black, font=\footnotesize, above]
      at (5.8, 2.05) {$p^{\ast}_{\Delta}(t)$};

    \draw[eqC] (\xa, \lA) -- (\xb, \lA);
    \draw[eqC] (\xb, \lB) -- (\xc, \lB);
    \draw[eqC] (\xc, \lC) -- (\xd, \lC);
    \draw[eqC] (\xd, \lD) -- (\xe, \lD);
    \draw[eqC] (\xe, \lE) -- (\xf, \lE);
    \draw[eqC] (\xf, \lF) -- (\xg, \lF);
    \draw[eqC] (\xg, \lG) -- (\xh, \lG);
    \draw[eqC] (\xh, \lH) -- (\xz, \lH);

    \node[black, font=\scriptsize, above]
      at ({(\xc + \xd)/2}, \lC) {$\lambda^{\ast}_{\Delta}(s_{3})$};

    \pgfmathsetmacro{\tannot}{(\xc + \xd)/2}
    \pgfmathsetmacro{\pannot}{1.35
      + 0.5 * (\tannot - \xc) / 1.25
      + 0.45 * (\tannot - \xc) * (\xd - \tannot) / 1.5625}
    \draw[<->, very thin, blue!70!black]
      ({\tannot}, {\pannot}) -- ({\tannot}, \lC);
    \node[blue!70!black, font=\scriptsize, right]
      at ({\tannot+0.05}, {(\pannot + \lC)/2}) {$c(s_{3}, t)$};

    \foreach \xL/\xR/\name in
      {{\xa}/{\xb}/{1}, {\xb}/{\xc}/{2}, {\xc}/{\xd}/{3},
       {\xd}/{\xe}/{4}, {\xe}/{\xf}/{5}, {\xf}/{\xg}/{6},
       {\xg}/{\xh}/{7}, {\xh}/{\xz}/{M}} {
      \draw[decorate, decoration={brace, mirror, raise=4pt},
            semithick, gray!70!black]
        (\xL, 0) -- (\xR, 0);
      \node[below, font=\scriptsize, black]
        at ({(\xL+\xR)/2}, -0.42)
        {$\Gamma^{\ast}_{\Delta}(s_{\name})$};
    }

    \node[below, font=\scriptsize, black]
      at (\xa, -1.0) {$t^{-}_{\Delta}(s_{1})$};
    \node[below, font=\scriptsize, black]
      at (\xz, -1.0) {$t^{+}_{\Delta}(s_{M})$};

    \draw[->, semithick, green!50!black] (\xa, 4.05) -- (\xz, 4.05);
    \node[above, font=\scriptsize, green!50!black]
      at ({(\xa+\xz)/2}, 4.05)
      {sorting: $s_{1} < s_{2} < \cdots < s_{M} \Rightarrow$
       $\Gamma^{\ast}_{\Delta}(s_{1}), \ldots, \Gamma^{\ast}_{\Delta}(s_{M})$
       contiguous and ordered};
  \end{tikzpicture}
  \caption{Sorting and toll pattern in the $\Delta$-DSO
    solution.}
  \label{fig:toll_pattern}
\end{figure}
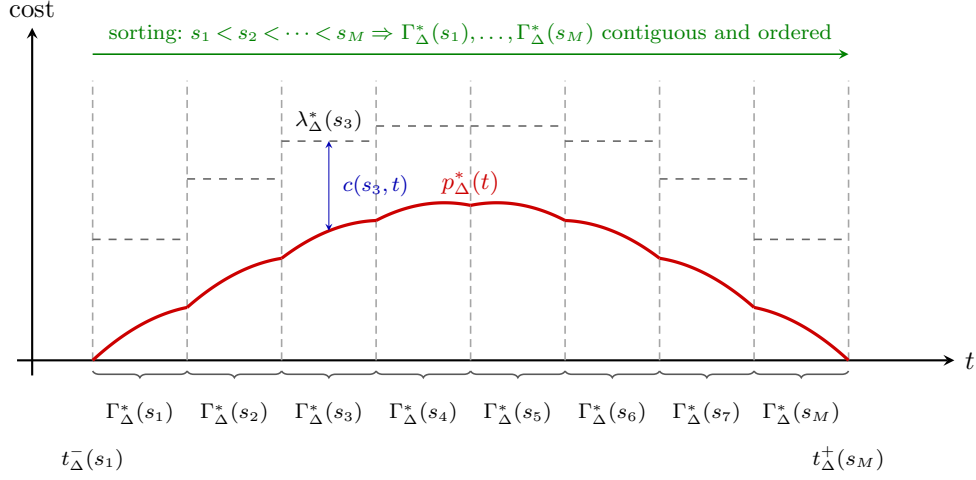

\subsection{Individual Vehicle Costs and True System Cost}
\label{subsec:individual_cost}

This subsection introduces two cost measures used throughout.
The \emph{individual decision cost} $C_\Delta(s;\theta)$ is what a
vehicle minimizes and includes the toll.
The \emph{social cost} $J$ aggregates schedule delay only and, since
the toll is a transfer between vehicles and the operator, is
unaffected by it.
The efficiency analysis of \cref{sec:efficiency} compares the
coarse-report social cost $J^{\mathrm{slot}}$ against its continuous
optimum $J^{\ast}$.
We take the within-slot placement of \cref{alg:dso}~(step~4) to be
uniform over the equispaced positions, which makes the slot mean of
the generalized cost the relevant decision cost. A deterministic or
worst-case placement would change $C_\Delta$ and the analysis
accordingly.

Under the optimal solution to the slot-based DSO problem, a vehicle
that reported slot $s$ arrives at some time within the assignment
interval
$\Gamma^{\ast}_\Delta(s) = [t^{-}_\Delta(s), t^{+}_\Delta(s)]$
of width $W_\Delta(s) \coloneqq t^{+}_\Delta(s) - t^{-}_\Delta(s)$.
The capacity constraint~\labelcref{eq:DSO_cap} forces the $d(s)$ vehicles
reporting slot $s$ to arrive at equally spaced times $1/\mu$ apart
inside $\Gamma^{\ast}_\Delta(s)$, so that $W_\Delta(s) = d(s)/\mu$.

A vehicle with true preferred arrival time $\theta_{n}$ that
ultimately arrives at a specific position $\tau \in
\Gamma^{\ast}_\Delta(s)$ incurs schedule delay cost
$c(\theta_{n}, \tau)$ and pays toll $p^{\ast}_{\Delta}(\tau)$,
for a generalized cost of $c(\theta_{n}, \tau) + p^{\ast}_{\Delta}(\tau)$
at the realized position.
By the uniform within-slot randomization (step~4 of the mechanism
in \cref{subsec:mechanism}), the realized position is drawn
uniformly from the $d(s)$ equispaced candidates, so the ex-ante
expected generalized cost is the discrete mean of
$c(\theta_{n}, \cdot) + p^{\ast}_{\Delta}(\cdot)$ over those
$d(s)$ positions.
We use this expected cost as the vehicle's decision cost.

\begin{lem}[Mean-cost evaluation within the slot]
  \label{lem:within_slot_mean}
  Under \cref{asm:schedule_cost}, the discrete
  mean of the schedule-plus-toll cost over the $d(s)$ equispaced
  positions of $\Gamma^{\ast}_\Delta(s)$ equals the continuous
  uniform mean over $\Gamma^{\ast}_\Delta(s)$ up to
  $\mathcal{O}(\Delta^2)$:
  \begin{align}
    \frac{1}{d(s)}\sum_{k=1}^{d(s)}
    \bigl[c(\theta, \tau_{k}) + p^{\ast}_\Delta(\tau_{k})\bigr]
    = C_\Delta(s;\, \theta) + \mathcal{O}(\Delta^2),
    \label{eq:discrete_to_continuous}
  \end{align}
  where $\tau_{k} = t^{-}_\Delta(s) + (k - 1/2)/\mu$ for
  $k = 1, \ldots, d(s)$ are the equispaced positions
  (the midpoints of the $d(s)$ subintervals of width $1/\mu$
  partitioning $\Gamma^{\ast}_\Delta(s)$), and
  \begin{align}
    C_\Delta(s;\, \theta) \coloneqq
    \frac{1}{W_\Delta(s)}
    \int_{\Gamma^{\ast}_\Delta(s)}
    \bigl[c(\theta, t) + p^{\ast}_\Delta(t)\bigr]\,\mathrm{d}t.
    \label{eq:individual_cost}
  \end{align}
\end{lem}
\begin{prf}
See \cref{app:proof_within_slot_mean}.
\end{prf}
The subsequent analysis uses the continuum surrogate $C_\Delta$
as the canonical decision cost, with the $\mathcal{O}(\Delta^2)$
discrepancy absorbed by the rate of the main results.
We write $C_\Delta(\hat{\theta}_{n};\, \theta_{n})$ for the cost of a
price-taking user of type $\theta_{n}$ that reports slot
$\hat{\theta}_{n}$, given the aggregate assignment
$q^{\ast}_{\Delta}$ and toll $p^{\ast}_{\Delta}$. Because a single
user has negligible mass, these aggregates do not depend on its
own report.
The aggregate counterpart, the \emph{true system cost}, sums
schedule delay at the true preferred arrival times:
\begin{align}
  J(\hat{\Vttheta} \mid \Vttheta,\, \Delta)
  \coloneqq
  \sum_{n \in \ClN}
  \frac{1}{t^{+}_{\Delta}(\hat{\theta}_{n}) - t^{-}_{\Delta}(\hat{\theta}_{n})}
  \int_{t^{-}_{\Delta}(\hat{\theta}_{n})}^{t^{+}_{\Delta}(\hat{\theta}_{n})}
  c(\theta_{n}, t) \,\mathrm{d}t.
  \label{eq:true_system_cost}
\end{align}
Toll revenue is lump-sum redistributed, so tolls cancel in $J$
but remain in $C_\Delta$.
$C_\Delta$ underlies the strategyproofness analysis of
\cref{sec:strategyproofness}, and $J$ underlies the efficiency
analysis of \cref{sec:efficiency}.

\section{Strategyproofness Analysis}
\label{sec:strategyproofness}

This section establishes the $\mathcal{O}(\Delta^{2})$ rate of
the $\varepsilon$-strategyproofness parameter
$\varepsilon^{\ast}(\Delta)$ under the slot-based DSO mechanism.
\cref{subsec:truthful} introduces strategyproofness and its
$\varepsilon$-relaxation as the relevant metrics.
\cref{subsec:regularity} defines the $\varepsilon$-strategyproofness
parameter $\varepsilon^{\ast}(\Delta)$ and bounds the per-vehicle
misreport gain $g_{\Delta}(k)$ at the adjacent slot ($k = \pm 1$)
quadratically.
\cref{subsec:eps_bound} extends the adjacent bound to arbitrary
$k$ and converts the result into the global bound on
$\varepsilon^{\ast}(\Delta)$.

\subsection{Strategyproofness and $\varepsilon$-Strategyproofness}
\label{subsec:truthful}

Strategyproofness is a mechanism design metric originating with
\citet{Clarke1971-hh,Groves1973-cu}:
it records whether any participant can reduce its cost by
unilaterally misreporting, taking the reports of others as
truthful.
The quantitative relaxation introduced by
\citet{Azevedo2018-yk,Budish2012-hg,Balseiro2024-jn} bounds the
worst such gain by a parameter $\varepsilon \geq 0$, giving a
graded measure of how close a mechanism is to strategyproofness.
This graded notion is an instance of approximate
strategyproofness, the family of relaxations surveyed by
\citet{LubinParkes2012-as}.
In our setting the question reduces to whether a vehicle with
preferred arrival time $\theta_{n}$ reports a slot other than
its \emph{truthful report} $r_{\Delta}(\theta_{n})$ defined by
the rounding map~\labelcref{eq:rounding}.
We formalize strategyproofness and its $\varepsilon$-relaxation
as follows.

\begin{dfn}[Strategyproofness in the large]
  \label{dfn:SP}
  The mechanism is \emph{strategyproof in the large} if, for every
  vehicle~$n$, every true preferred arrival time~$\theta_{n}$, and
  every report $\hat{\theta}_{n} \in \ClS_\Delta$,
  \begin{align}
    C_{\Delta}(r_{\Delta}(\theta_{n});\, \theta_{n})
    \leq
    C_{\Delta}(\hat{\theta}_{n};\, \theta_{n}).
    \label{eq:SP}
  \end{align}
  Here $C_{\Delta}(\cdot;\, \theta_{n})$ is evaluated at the aggregate
  assignment and toll induced by the population, which a price-taking
  user takes as given.
  This is the nonatomic, price-taking notion (in the sense of
  \citet{Azevedo2018-yk}) rather
  than dominant-strategy incentive compatibility for a finite
  population.
  We refer to it simply as \emph{strategyproofness} in the remainder
  of the paper.
\end{dfn}

The continuous-report limit is the relevant benchmark. When slots
shrink to points, the toll restores exact strategyproofness.
Write $C_{0}(\hat\theta;\theta) \coloneqq \lim_{\Delta \to 0}
C_{\Delta}(\hat\theta;\theta)$ for the cost of a price-taking
type-$\theta$ user reporting $\hat\theta$ under the continuous
assignment and toll.

\begin{pro}[Exact strategyproofness in the continuous limit]
  \label{prop:continuous_SP}
  Under \cref{asm:schedule_cost,asm:demand_density}, in the
  price-taking limit the continuous mechanism is strategyproof:
  for every type $\theta$ and every report $\hat\theta \in \ClS$,
  \begin{align}
    C_{0}(\theta;\theta) \;\leq\; C_{0}(\hat\theta;\theta),
    \label{eq:continuous_SP}
  \end{align}
  with strict inequality for $\hat\theta \neq \theta$ under the
  strict convexity implied by~(C4) of \cref{asm:schedule_cost}.
\end{pro}
\begin{prf}
  Reporting $\hat\theta$ assigns the user to the time
  $\tau(\hat\theta) = \arg\min_{t}\{c(\hat\theta, t) +
  p^{\ast}_{0}(t)\}$ at which slot $\hat\theta$ is served in the
  tolled optimum~\labelcref{eq:KKT_flow}, so its cost is
  $c(\theta, \tau(\hat\theta)) + p^{\ast}_{0}(\tau(\hat\theta))$.
  As $\hat\theta$ ranges over $\ClS$, $\tau(\hat\theta)$ sweeps the
  active window, so the user minimizes
  $c(\theta, t) + p^{\ast}_{0}(t)$ over it. By the same tolled
  optimality applied to slot $\theta$, together with the strict
  convexity implied by~(C4), its unique minimizer is $\tau(\theta)$, attained by
  truthful reporting.
\end{prf}

\cref{prop:continuous_SP} is the marginal-cost (Pigovian)
decentralization of the dynamic system optimum
\citep{Vickrey1969-ic, Arnott1990-pc, Arnott1993-lu}, recast as a
direct-revelation mechanism. Equivalently, the anonymous toll
implements the efficient assignment truthfully by the single-crossing
argument standard in mechanism design
\citep{Clarke1971-hh, Groves1973-cu}, and is strategyproof in the
large \citep{Azevedo2018-yk}.
The toll is thus the device that secures strategyproofness in the
ideal. Coarse reporting is the sole obstruction.

Exact strategyproofness is unattainable under coarse reports:
the rounding map $r_{\Delta}$ is discontinuous across slot
boundaries, so a vehicle near a boundary can always gain by
misreporting to the adjacent slot.
We measure approximation to strategyproofness by the smallest
$\varepsilon$ that bounds every unilateral deviation gain, and
formally introduce $\varepsilon$-strategyproofness as follows.

\begin{dfn}[$\varepsilon$-strategyproofness]
  \label{dfn:eps_SP}
  The mechanism is $\varepsilon$-strategyproof if, for every
  $n, \theta_{n}, \hat{\theta}_{n}$,
  \begin{align}
    C_{\Delta}(r_{\Delta}(\theta_{n});\, \theta_{n})
    \leq
    C_{\Delta}(\hat{\theta}_{n};\, \theta_{n})
    + \varepsilon.
    \label{eq:eps_SP}
  \end{align}
\end{dfn}

The smallest $\varepsilon$ satisfying Inequality~\labelcref{eq:eps_SP}, denoted
$\varepsilon^{\ast}(\Delta)$, depends on the slot width $\Delta$
and is the central quantity of this section.
Setting $\varepsilon = 0$ recovers \cref{dfn:SP}.

\subsection{Quadratic Bound on the Misreport Gain}
\label{subsec:regularity}
\label{subsec:misreport_bound}%

We therefore characterize how $\varepsilon^{\ast}(\Delta)$
scales with $\Delta$.
From \cref{dfn:eps_SP}, the tightest $\varepsilon$ for which the
mechanism is $\varepsilon$-strategyproof under coarsening
$\Delta$ is
\begin{align}
  \varepsilon^{\ast}(\Delta)
  \coloneqq
  \sup_{n, \theta_{n}, \hat{\theta}_{n}}
  \Bigl\{
    C_{\Delta}(r_{\Delta}(\theta_{n});\, \theta_{n})
    - C_{\Delta}(\hat{\theta}_{n};\, \theta_{n})
  \Bigr\}_{+},
  \label{eq:eps_Delta_def}
\end{align}
with $\{x\}_{+} \coloneqq \max\{0, x\}$.

To do so constructively, we focus on the cost reduction obtained
by misreporting $k$ slots away from the truthful slot
$r_{\Delta}(\theta_{n})$:
\begin{align}
  g_{\Delta}(k;\, \theta_{n})
  \;\coloneqq\;
  C_{\Delta}(r_{\Delta}(\theta_{n});\, \theta_{n})
  - C_{\Delta}(r_{\Delta}(\theta_{n}) + k\Delta;\, \theta_{n}),
  \label{eq:g_def}
\end{align}
abbreviated $g_{\Delta}(k)$ when the context is clear.
For the adjacent slot case $k = \pm 1$, the schedule cost curvature
(\cref{asm:schedule_cost}(C4)) and the peak rate
$\nu^{\mathrm{peak}}/\mu$ bound the gain quadratically:

\begin{lem}[Quadratic bound on the adjacent slot misreport gain]
  \label{lem:eps_quadratic}
  Under \cref{asm:schedule_cost,asm:demand_density}, for every
  $n, \theta_{n}$,
  \begin{align}
    \bigl|g_{\Delta}(\pm 1;\, \theta_{n})\bigr|
    \leq \ddot{c}_{\max}\,(\nu^{\mathrm{peak}}/\mu)\,\Delta^{2},
    \label{eq:eps_quadratic}
  \end{align}
  with $\ddot{c}_{\max}, \nu^{\mathrm{peak}}$ from
  \cref{asm:schedule_cost,asm:demand_density}.
\end{lem}
\begin{prf}
See \cref{app:proof_eps_quadratic}.
\end{prf}

\subsection{$\varepsilon$-strategyproofness bound}
\label{subsec:eps_bound}

This subsection extends the adjacent slot bound of
\cref{lem:eps_quadratic} to arbitrary misreport distances
$k \in \mathbb{Z}$ and converts the resulting uniform bound on
$g_{\Delta}(k)$ into the $\ClO(\Delta^{2})$ bound on
$\varepsilon^{\ast}(\Delta)$ stated in
\cref{thm:eps_quadratic_global}.
The conversion is direct. By \cref{eq:eps_Delta_def},
$\varepsilon^{\ast}(\Delta) = \sup_{n, \theta_{n}, k}
\{g_{\Delta}(k;\, \theta_{n})\}_{+}$,
so a uniform $\ClO(\Delta^{2})$ bound on $g_{\Delta}(k)$ across
all $k$ transfers to $\varepsilon^{\ast}(\Delta)$.
The extension to all $k$ combines \cref{lem:eps_quadratic} with
the discrete concavity of $g_{\Delta}$
(\cref{lem:g_concavity}, proved in \cref{app:g_concavity}):
concavity bounds the maximiser $k^{\ast}$ of $g_{\Delta}$ by an
integer $\bar k$ independent of $\Delta$, and telescoping over
$j = 0, \ldots, k^{\ast}-1$ gives
$g_{\Delta}(k^{\ast}) = \mathcal{O}(\Delta^{2})$.

This extension requires a peakedness and curvature condition on
the demand and cost primitives, stated formally as~\labelcref{eq:cc_regime}
in \cref{thm:eps_quadratic_global}.
In the nonatomic (price-taking) limit a single user has
negligible mass, so its deviation does not perturb the slot loads,
and the bound holds for every $\Delta \in (0, \Delta_{\max}]$, with
$\Delta_{\max}$ determined by the higher-order smoothness remainder
and the telescoping slot count (\cref{app:regime}).

\begin{thm}[Quadratic upper bound on the $\varepsilon$-strategyproofness
  parameter]
  \label{thm:eps_quadratic_global}
  Set $\eta \coloneqq 1 - \mu/\nu^{\mathrm{peak}}$.
  Under \cref{asm:schedule_cost,asm:demand_density}, the
  peakedness and curvature condition
  \begin{align}
    \eta \in (0, \tfrac{1}{2}),
    \qquad
    \frac{\ddot{c}_{\min}}{\ddot{c}_{\max}}
    \;>\; \eta + \eta^{2},
    \label{eq:cc_regime}
  \end{align}
  and a smoothness threshold
  $L_{\nu} \leq c_{\nu}\,\mu/\dot{c}_{\max}$ on the demand prior
  for a positive constant $c_{\nu}$ independent of $\Delta$, there
  exist constants
  $C_{\varepsilon}, \Delta_{\max} > 0$ depending only
  on the primitives $(\nu, \mu, c)$ but not on $\Delta$ such that
  \begin{align}
    \varepsilon^{\ast}(\Delta)
    \;\leq\;
    C_{\varepsilon}\,\Delta^{2}
    \qquad
    \text{for all } \Delta \in (0, \Delta_{\max}].
    \label{eq:eps_quadratic_global}
  \end{align}
\end{thm}
\begin{prf}
See \cref{app:proof_epsilon}.
\end{prf}

\cref{thm:eps_quadratic_global} states that, under the peakedness
and curvature condition~\labelcref{eq:cc_regime},
the worst single-vehicle misreport incentive shrinks as the
square of the slot width. Halving $\Delta$ quarters
$\varepsilon^{\ast}(\Delta)$.
This is one order tighter than the $\mathcal{O}(\Delta)$ rate that
pure Lipschitz continuity of the cost and toll would give.
The sharpening comes from the KKT condition~\labelcref{eq:KKT_flow}
of the slot-based DSO problem, which cancels the first-order
term.

The proof in \cref{app:proof_epsilon} proceeds in two steps.
The adjacent slot bound of \cref{lem:eps_quadratic} extends to
the slot level marginal gain
$D(j) \coloneqq C_{\Delta}(s^j;\theta_{n})
- C_{\Delta}(s^{j+1};\theta_{n})$ as a quadratic seed
$D(j) \leq \ddot{c}_{\max}(\nu^{\mathrm{peak}}/\mu)(j+1)\Delta^{2}$,
and the discrete concavity of $g_{\Delta}$
(\cref{lem:g_concavity}) bounds the worst misreport distance by
a $\Delta$-independent constant $\bar k$.
Telescoping the seed over $\bar k$ steps yields the global bound.
The explicit form of $C_{\varepsilon}$ and
the sufficiency analysis of the peakedness and curvature
condition~\labelcref{eq:cc_regime} are deferred to
\cref{app:regime}.

The curvature condition in~\labelcref{eq:cc_regime} is the price of
admitting asymmetric or non-quadratic schedule delay costs.
For a symmetric quadratic cost
$c(\theta, t) = \beta(\theta - t)^{2}$ the curvature is constant,
$\ddot{c}_{\min} = \ddot{c}_{\max}$, so the curvature condition
$\ddot{c}_{\min}/\ddot{c}_{\max} > \eta + \eta^{2}$ holds
automatically (as $\eta + \eta^{2} < 3/4 < 1$ on
$\eta \in (0, \tfrac{1}{2})$), and the condition reduces to the single
peakedness condition $\nu^{\mathrm{peak}}/\mu < 2$.

\section{Efficiency Analysis}
\label{sec:efficiency}

This section bounds the expected efficiency loss
$L^{\ast}(\Delta)$ of the slot-based mechanism relative to the
continuous report DSO benchmark, in parallel with the
$\varepsilon$-strategyproofness analysis of
\cref{sec:strategyproofness}.
\cref{subsec:eff_linear} defines $L^{\ast}(\Delta)$ in terms of
the continuous DSO benchmark and the slot-based true social cost,
then derives a baseline rate $\mathcal{O}(\Delta)$ from the
Lipschitz continuity of the schedule delay cost alone
(\cref{asm:schedule_cost}(C3)).
\cref{subsec:eff_quadratic} sharpens it to
$\mathcal{O}(\Delta^{2})$ by exploiting two structural
cancellations in the $\Delta$-DSO solution.

\subsection{Definition of the Efficiency Loss}
\label{subsec:eff_linear}

We compare the slot-based true social cost under truthful
reporting against the continuous DSO benchmark.
The former is the true system cost
$J(\hat{\Vttheta} \mid \Vttheta,\, \Delta)$ of
\cref{eq:true_system_cost} specialized to truthful reports
$\hat\Vttheta = \tilde\Vttheta^{\Delta}
\coloneqq (r_{\Delta}(\theta_{1}), \ldots, r_{\Delta}(\theta_{N}))$,
which we abbreviate to
$J^{\mathrm{slot}}(\Vttheta;\, \Delta)
\coloneqq J(\tilde\Vttheta^{\Delta} \mid \Vttheta,\, \Delta)$
when the dependence on the report profile is fixed at the
truthful one.
The benchmark is the optimal social cost without coarsening,
\begin{align}
  J^{\ast}(\Vttheta) \coloneqq
  \min_{\{\tau_{n}\}_{n \in \ClN}}
  \sum_{n \in \ClN} c(\theta_{n},\, \tau_{n})
  \quad\text{s.t.\ the bottleneck capacity at rate $\mu$,}
  \label{eq:J_star_def}
\end{align}
i.e., the $\Delta \to 0$ limit of $[\Delta\text{-DSO}]$ evaluated
at the unrounded profile $\Vttheta$.
Under binding capacity the optimal assignment sorts vehicles by
preferred time, $\tau(\theta) = t_{\mathrm{start}} + (N/\mu)F(\theta)$
with $F(\theta) = \tfrac{1}{N}\int_{\theta^{-}}^{\theta}\nu$ the
normalized prior. Differentiating gives the flow identity
$\nu(\theta)\,\mathrm{d}\theta = \mu\,\mathrm{d}\tau$.
The social cost then has two equivalent expressions,
\begin{align}
  J = \sum_{n \in \ClN} c(\theta_{n}, \tau_{n})
  = \int_{\ClS} \nu(\theta)\,c(\theta, \tau(\theta))\,\mathrm{d}\theta
  = \int_{\ClT} \mu\,c(\theta(t), t)\,\mathrm{d}t,
  \label{eq:J_flow_identity}
\end{align}
the same quantity weighted along the preferred-time axis (middle) or
the departure-time axis (right).
The expected efficiency loss is
\begin{align}
  L^{\ast}(\Delta)
  \coloneqq
  \int_{\ClS^N}
  \bigl[J^{\mathrm{slot}}(\Vttheta;\, \Delta) - J^{\ast}(\Vttheta)\bigr]
  \mathrm{d}\nu(\Vttheta)
  \geq 0,
  \label{eq:eff_loss_def}
\end{align}
where the integral weights the per-type loss by the prior $\nu$
(\cref{asm:demand_density}). Under the deterministic representative
assumption~(D3), $L^{\ast}(\Delta)$ is the loss at this representative
profile rather than an average over random realizations, and it scales
with the total mass $N$.

\subsection{Quadratic Upper Bound}
\label{subsec:eff_quadratic}

Under the basic assumptions plus binding capacity, the efficiency
loss is in fact $\mathcal{O}(\Delta^{2})$ in the slot width, one
order below the $\mathcal{O}(\Delta)$ rate that Lipschitz
continuity of the schedule delay cost alone would give, and
without invoking the curvature condition~\labelcref{eq:cc_regime} of
\cref{thm:eps_quadratic_global}.

\begin{thm}[Quadratic upper bound on the expected efficiency
  loss]
  \label{thm:eff_loss_quadratic}
  Under \cref{asm:schedule_cost,asm:demand_density} with the
  binding capacity condition in force, there exists a constant
  $C_{L} > 0$ depending only on the primitives $(\nu, \mu, c)$
  but not on $\Delta$ such that
  \begin{align}
    L^{\ast}(\Delta) \;\leq\; N\,C_{L}\,\Delta^{2}
    \qquad
    \text{for all } \Delta > 0.
    \label{eq:eff_loss_quadratic}
  \end{align}
\end{thm}
\begin{prf}
See \cref{app:proof_eff_quadratic}.
\end{prf}

The improvement over the linear baseline rests on two structural
properties of the $\Delta$-DSO assignment.
First, the slot level rank is preserved between the continuous
DSO and the $\Delta$-DSO. The rounding map (\cref{eq:rounding})
respects the order of $\theta_{n}$ across distinct slots, and
within a slot the tie break is immaterial because vehicles share
the same report and \cref{eq:DSO_obj,eq:DSO_dem} depend only on
the slot count.
Second, the within slot arrivals are equidistant around the
centroid $\bar\tau_{i} \coloneqq
(t^{-}_\Delta(s_{i}) + t^{+}_\Delta(s_{i}))/2$, since the
capacity binding KKT condition forces the $d(s_{i})$ vehicles to
arrive at $1/\mu$ intervals inside $\Gamma^{\ast}_\Delta(s_{i})$
(\cref{subsec:individual_cost}).
Combining the two cancels the first order shift between
continuous and slot assignments, leaving only the second order
residual $\mathcal{O}(\Delta^{2})$ of the Taylor expansion.

\cref{thm:eff_loss_quadratic} is the efficiency analogue of
\cref{thm:eps_quadratic_global}. Both yield the
$\mathcal{O}(\Delta^{2})$ rate from the same cancellation,
namely the active period KKT condition together with the
equidistant within slot placement of arrivals.
The two bounds differ in their scope.
\cref{thm:eps_quadratic_global} requires the peakedness and
curvature condition~\labelcref{eq:cc_regime}, because the
strategyproofness analysis must control the maximizer of a discrete
misreport gain sequence through its concavity.
\cref{thm:eff_loss_quadratic} holds for every $\Delta > 0$ under
the basic assumptions. The efficiency cancellation invokes only
the first order condition for the continuous DSO start time,
which is available whenever the $\Delta$-DSO inherits the sorting
property of the continuous one (\cref{lem:FIFW} under full
congestion).

The constant $C_{L}$ admits the explicit form
\begin{align}
  C_{L}
  \;=\;
  \underbrace{\tfrac{1}{8}\,\ddot{c}_{\max}
    \left(1 + \tfrac{\nu^{\mathrm{peak}}}{\mu}\right)^{\!2}}_{\text{Taylor remainder}}
  \;+\;
  \underbrace{\tfrac{1}{4}\,\ddot{c}_{\max}
    \left(1 + \tfrac{\nu^{\mathrm{peak}}}{\mu}\right)
    \tfrac{\nu^{\mathrm{peak}}}{\mu}}_{\text{within slot cross sum}}
  \;+\;
  \underbrace{\mathcal{O}\!\left(\tfrac{\dot{c}_{\max}\,L_{\nu}}{\mu}\right)}_{\text{slot mean discrepancy}},
  \label{eq:C_L_explicit}
\end{align}
where $L_{\nu}$ is the Lipschitz constant of $\nu$, and all three
contributions are independent of $\Delta$ under
\cref{asm:demand_density}.
The first two terms come from the Taylor expansion in the proof
of \cref{thm:eff_loss_quadratic}, and
the third absorbs the smoothness contribution from the slot
mean discrepancy of \cref{lem:within_slot_mean}.

\cref{thm:eff_loss_quadratic} bounds the efficiency loss at the
truthful profile $\tilde\Vttheta^{\Delta}$.
The realized reporting profile need not be truthful, since vehicles
near a slot boundary may strictly prefer an adjacent slot.
Under the toll, however, the realized (equilibrium) profile stays
close to truthful. The discrete concavity of the misreport gain
(\cref{lem:g_concavity}) confines any profitable deviation to an
adjacent slot with gain at most
$\varepsilon^{\ast}(\Delta) = \ClO(\Delta^2)$, so the induced
assignment shifts by $\ClO(\Delta)$ only for vehicles within
$\ClO(\Delta)$ of a boundary.
The first-order effect of these shifts on the social cost cancels
through the same start-time first-order condition and within-slot
symmetry that underlie \cref{thm:eff_loss_quadratic}, leaving an
$\ClO(\Delta^2)$ residual.
A direct numerical evaluation confirms this. Under best response
(price-taking equilibrium) reporting, the realized efficiency loss
$J(\sigma) - J^{\ast}$ has a log-log slope of $1.99$ in $\Delta$,
matching the truthful profile rate, while the fraction of deviating
vehicles stays below $4\%$ and decreases as $\Delta$ shrinks.
This argument is heuristic. A full analytic characterization of the
coarse reporting equilibrium, a price of anarchy under coarse
reporting, is left for future work.
Without the toll the misreport incentive is $\ClO(1)$
(\cref{prop:NP_lower_bound}), the realized profile departs from the
truthful one by an $\ClO(1)$ amount, and this argument no longer
applies. The toll is what keeps the realized assignment close to the
efficient one.

\section{Numerical Experiments}
\label{sec:numerical}

This section evaluates $L^{\ast}(\Delta)$
(\cref{thm:eff_loss_quadratic}) and $\varepsilon^{\ast}(\Delta)$
(\cref{thm:eps_quadratic_global}) numerically and interprets
the magnitudes at operationally relevant slot widths.
\cref{subsec:numerical_setup} fixes the common baseline parameter
set used throughout.
\cref{subsec:numerical_eff,subsec:numerical_SP} evaluate
$L^{\ast}(\Delta)$ and $\varepsilon^{\ast}(\Delta)$ on that
baseline.
\cref{subsec:robustness} sweeps two axes
(\cref{tab:experimental_design}) that probe the peakedness and
curvature condition~\labelcref{eq:cc_regime}
of \cref{thm:eps_quadratic_global}.

\begin{table}[tbp]
  \caption{Experimental design.}
  \label{tab:experimental_design}
  \centering
  \begin{tabular}{lll}
    \toprule
    Axis & Baseline & Sweep \\
    \midrule
    Cost asymmetry $\gamma/\beta$
      & $2$ & $\{1,\, 1.5,\, 2,\, 3,\, 5\}$ \\
    Demand prior $\nu$
      & Triangular
      & Uniform, Triangular, Beta(2,5), Beta(5,2) \\
    \bottomrule
  \end{tabular}
\end{table}

\subsection{Setup}
\label{subsec:numerical_setup}
We adopt the asymmetric quadratic schedule delay cost function
\begin{align}
  c(\theta, t) = \beta\,(\theta - t)^2 \cdot \mathbf{1}_{t \leq \theta}
  + \gamma\,(t - \theta)^2 \cdot \mathbf{1}_{t > \theta},
  \qquad \gamma > \beta > 0,
  \label{eq:quadratic_cost}
\end{align}
which satisfies conditions (C1)--(C3) of \cref{asm:schedule_cost}
with Lipschitz constant
$\dot{c}_{\max} = 2\gamma T$ on $(\theta, t) \in \ClS \times \ClT$.
The baseline parameter values are listed in
\cref{tab:baseline_parameters}.

\begin{table}[tbp]
  \caption{Baseline parameter values.}
  \label{tab:baseline_parameters}
  \centering
  \begin{tabular}{l l p{0.45\textwidth}}
    \toprule
    Symbol & Value & Description \\
    \midrule
    $N$ & $720$ & Number of vehicles \\
    $\mu$ & $1.5$\,veh\,min$^{-1}$ & Bottleneck capacity \\
    $S$ & $360$\,min & Support width of $\ClS$ (6 hours) \\
    $T$ & $N/\mu = 480$\,min & Assignment window length \\
    $\nu$ & Symmetric triangular on $\ClS$
      & Preferred arrival density (peak at center,
        Lipschitz, satisfying \cref{asm:demand_density}) \\
    $\beta,\, \gamma$ & $0.3/3600,\, 0.6/3600$
      & Schedule delay coefficients (early, late) \\
    $\Delta$ & $T/M$, $M = 2, 3, \ldots$
      & Slot width (with $T/\Delta$ integer) \\
    \bottomrule
  \end{tabular}
\end{table}
The baseline macroscopic load is $N/(S\mu) \approx 1.33 > 1$, so
capacity binds throughout the active period.
The triangular prior gives $\nu^{\mathrm{peak}}/\mu \approx 2.66$,
above the peakedness condition $\nu^{\mathrm{peak}}/\mu < 2$
(the first inequality of~\labelcref{eq:cc_regime}), while the uniform
prior ($\nu^{\mathrm{peak}}/\mu \approx 1.33$) sits inside it.
The robustness sweep of \cref{subsec:robustness} shows the
$\mathcal{O}(\Delta^2)$ rate persists beyond the cut-off.
We restrict $\Delta$ to $M = T/\Delta$ integer.
\cref{fig:setup_shapes} shows the asymmetric quadratic schedule
delay cost (left) and the four preferred-arrival-time priors used
in the experiments (right, with the triangular baseline and the
others introduced in the robustness checks of
\cref{subsec:robustness}).

\begin{figure}[tbp]
  \centering
  \begin{minipage}[t]{0.48\textwidth}
    \centering
    \includegraphics[width=\textwidth]{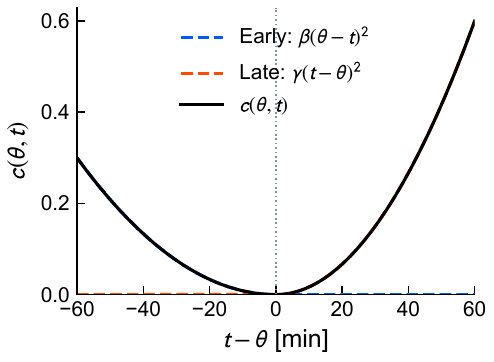}
    \subcaption{Schedule delay cost.}
    \label{fig:schedule_cost}
  \end{minipage}\hfill
  \begin{minipage}[t]{0.48\textwidth}
    \centering
    \includegraphics[width=\textwidth]{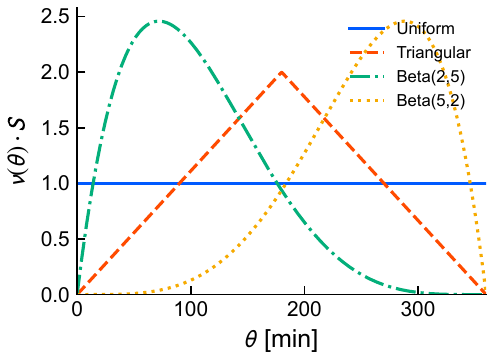}
    \subcaption{Preferred-arrival-time densities.}
    \label{fig:density_shapes}
  \end{minipage}
  \caption{Cost and demand shapes used in the experiments.}
  \label{fig:setup_shapes}
\end{figure}

For each $\Delta$, we construct $\ClS_\Delta$, numerically
optimize $\Gamma^{\ast}_\Delta$ and the congestion start time,
and obtain $\lambda^{\ast}_{\Delta}, p^{\ast}_{\Delta}$ from
\cref{eq:lambda_init,eq:lambda_recursion}.
The continuous-time DSO benchmark uses
$t^{\ast}(\theta) = t_{\mathrm{start}} + NF(\theta)/\mu$ with
$t_{\mathrm{start}}$ optimized. For the triangular prior
$J^{\ast} \approx 363.0$.
All computations are in Python (SciPy).

\subsection{Precise Evaluation of Efficiency Loss}
\label{subsec:numerical_eff}
We compute the true total cost $J^{\mathrm{slot}}(\Delta)$ by
weighting each preferred arrival time $\theta$ by the prior density
$\nu$ and averaging the schedule delay cost over the assignment
interval into which its slot is mapped:
\begin{align}
  J^{\mathrm{slot}}(\Delta) =
  \sum_{i=1}^{M} \int_{s_{i} - \Delta/2}^{s_{i} + \Delta/2}
  \nu(\theta)\,
  \frac{1}{W_{i}} \int_{t^{-}_\Delta(s_{i})}^{t^{+}_\Delta(s_{i})}
  c(\theta, t)\,\mathrm{d}t\,\mathrm{d}\theta,
\end{align}
where $W_{i} = d(s_{i})/\mu$ is the assignment interval width and
the per-slot weight is
$\int_{s_{i}-\Delta/2}^{s_{i}+\Delta/2}\nu(\theta)\,\mathrm{d}\theta
= d(s_{i})$, generally nonuniform across slots.
Within each slot the realized arrival is uniform over
$\Gamma^{\ast}_\Delta(s_{i})$.

\cref{fig:efficiency_loss} plots the numerical $L^{\ast}(\Delta)$
against the quadratic envelope $N\,C_{L}\,\Delta^{2}$ of
\cref{thm:eff_loss_quadratic}.
A log-log fit yields a slope of $2.00$, confirming the
$\ClO(\Delta^{2})$ rate and lying well below a linear
$\ClO(\Delta)$ rate.
Relative to the continuous-DSO optimum $J^{\ast}$, the loss is
$0.25\%$ at $\Delta = 6$\,min, $1.5\%$ at $15$\,min, and $6.1\%$ at
$30$\,min, so efficiency stays near-optimal well above the
sub-minute range.

The curvature jump of this cost at $\theta = t$
($2\beta \to 2\gamma$) lies on a measure-zero set and enters the
proof of \cref{thm:eff_loss_quadratic} only inside an integral over
$\theta$, so it leaves the $\ClO(\Delta^2)$ order intact.

\begin{figure}[tbp]
  \centering
  \includegraphics[width=0.42\textwidth]{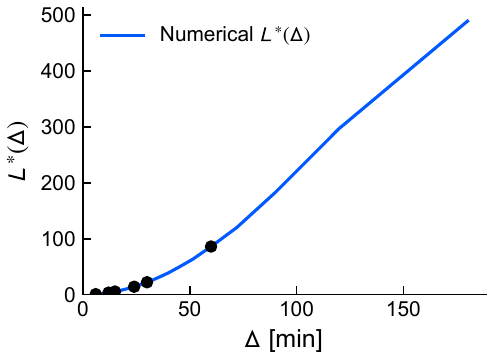}
  \caption{Efficiency loss $L^{\ast}(\Delta) = J^{\mathrm{slot}}(\Delta)
  - J^{\ast}$ versus slot width $\Delta$ under the baseline
  triangular preferred-arrival prior of \cref{tab:baseline_parameters}.
  Markers show numerical values and the solid line shows the
  theoretical $\mathcal{O}(\Delta^{2})$ bound.}
  \label{fig:efficiency_loss}
\end{figure}

\subsection{Numerical Evaluation of $\varepsilon$-Strategyproofness Parameter}
\label{subsec:numerical_SP}
To compute $\varepsilon^{\ast}(\Delta)$ numerically, for each
$\theta \in [\theta^{-}, \theta^{+}]$ we compare the truthful
reporting cost $C_\Delta(r_\Delta(\theta);\, \theta)$ with the
misreporting cost $C_\Delta(s';\, \theta)$ for every slot
$s' \in \ClS_\Delta$.
This evaluation matches the unilateral-deviation
definition (\cref{dfn:eps_SP}). All other vehicles report
truthfully and the DSO solution is computed under that profile.
The single-vehicle misreport perturbs the slot count $d(s)$ by
$\pm 1$. This finite population effect, absent in the nonatomic
(price-taking) analysis of \cref{sec:strategyproofness}, is
negligible for the parameter ranges tested ($d = 10$--$60$), as is
the $\mathcal{O}(1/\mu^2)$ quadrature floor of
\cref{lem:within_slot_mean}, which stays below the observed
$\mathcal{O}(\Delta^2)$ signal over this range.
The maximum cost reduction achievable by misreporting is:
\begin{align}
  \varepsilon^{\ast}(\Delta) = \sup_{\theta,\, s'}
  \max\bigl\{0,\;
    C_\Delta(r_\Delta(\theta);\, \theta) - C_\Delta(s';\, \theta)
  \bigr\}.
\end{align}
In the numerical evaluation, the search over $s'$ covers
\emph{all} slots in $\ClS_\Delta$, not only adjacent ones.
The optimal misreport is empirically always the adjacent slot,
which agrees with the discrete concavity of the gain sequence
established in the appendix. Under the curvature
condition~\labelcref{eq:cc_regime}, $g_{\Delta}(k)$ is strictly discretely concave in
$k$, so its maximiser $k^{\ast}$
over integer $k \ne 0$ is the slot at distance one from the
truthful slot.
The numerically computed $\varepsilon^{\ast}(\Delta)$ exhibits the
$\ClO(\Delta^2)$ rate. The prefactor
$\varepsilon^{\ast}(\Delta)/\Delta^2$ stays approximately constant
across the tested range (median $\approx 1.4 \times 10^{-5}$),
more than one and a half orders of magnitude below the adjacent-slot bound
$\ddot{c}_{\max}\,(\nu^{\mathrm{peak}}/\mu) \approx
8.9 \times 10^{-4}$ of \cref{lem:eps_quadratic}, so the rate is
tight in scaling while the constant of
\cref{thm:eps_quadratic_global} is conservative.
To assess practical significance, we normalize
$\varepsilon^{\ast}(\Delta)$ by the slot-averaged equilibrium cost
$\bar{\lambda}(\Delta) =
M^{-1}\sum_{i=1}^{M} \lambda^{\ast}_{\Delta}(s_{i})$
actually experienced by a vehicle:
\begin{align}
  \text{relative misreporting gain} =
  \frac{\varepsilon^{\ast}(\Delta)}{\bar{\lambda}(\Delta)}.
  \label{eq:relative_epsilon}
\end{align}
For an operator with indifference band $\omega$,
$\varepsilon$-strategyproofness is effectively satisfied whenever
$\varepsilon^{\ast}(\Delta)/\bar{\lambda}(\Delta) < \omega$
(\cref{fig:epsilon_relative}).

\begin{figure}[tbp]
  \centering
  \includegraphics[width=0.4\textwidth]{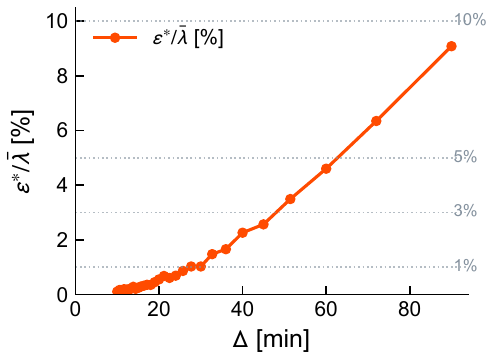}
  \caption{Relative misreporting gain
    $\varepsilon^{\ast}(\Delta)/\bar{\lambda}(\Delta)$ versus slot
    width $\Delta$, with $\bar{\lambda}(\Delta)$ the average
    equilibrium cost over slots. A smaller ratio indicates that the
    strategic incentive is small compared with the cost the
    vehicle already pays.}
  \label{fig:epsilon_relative}
\end{figure}

Operationally, the relative misreport gain
$\varepsilon^{\ast}(\Delta)/\bar\lambda(\Delta)$ stays below
$1\%$ for every $\Delta$ at or below $15$\,min and below $0.1\%$
for $\Delta \leq 6$\,min (\cref{fig:epsilon_relative}), so any
operator whose indifference band on the misreport gain exceeds
$\omega = 1\%$ accommodates calendar-friendly slot widths
without further refinement.

\subsection{Robustness Checks}
\label{subsec:robustness}
We probe the sufficiency of the peakedness and curvature
condition~\labelcref{eq:cc_regime} of
\cref{thm:eps_quadratic_global} by sweeping the two primitives
that directly enter it, namely the cost asymmetry
$\gamma/\beta$ (which controls the curvature ratio
$\ddot{c}_{\min}/\ddot{c}_{\max} = 1/(\gamma/\beta)$) and the
demand prior shape (which controls the peakedness
$\nu^{\mathrm{peak}}/\mu$).
\cref{fig:sensitivity_ratio,fig:distribution_sensitivity}
report the resulting $L^{\ast}(\Delta)$ and
$\varepsilon^{\ast}(\Delta)$ curves. The
$\mathcal{O}(\Delta^{2})$ rate is preserved on \emph{both}
performance metrics across both sweeps, including parameter
values that violate the sufficient condition~\labelcref{eq:cc_regime}.
The right panel of \cref{fig:sensitivity_ratio} plots the largest
slot width meeting a relative misreport tolerance~$\omega$,
\begin{align}
  \Delta^{\ast}(\omega) \coloneqq
  \max\{\Delta :
    \varepsilon^{\ast}(\Delta)/\bar{\lambda}(\Delta) \le \omega\},
  \label{eq:Delta_star}
\end{align}
which converts the asymptotic rate into an operational slot width
recommendation for a given indifference band~$\omega$.

\begin{figure}[tbp]
  \centering
  \begin{subfigure}[t]{0.95\textwidth}
    \centering
    \includegraphics[width=\linewidth]{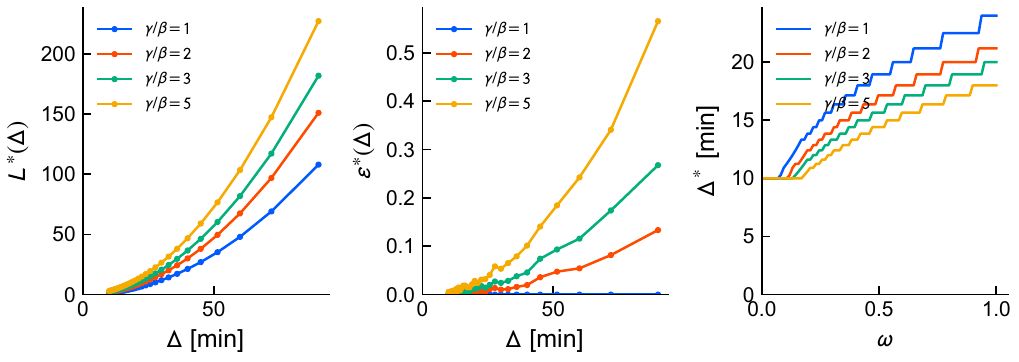}
    \subcaption{Cost asymmetry $\gamma/\beta$.}
    \label{fig:sensitivity_ratio}
  \end{subfigure}\\[0.5ex]
  \begin{subfigure}[t]{0.95\textwidth}
    \centering
    \includegraphics[width=\linewidth]{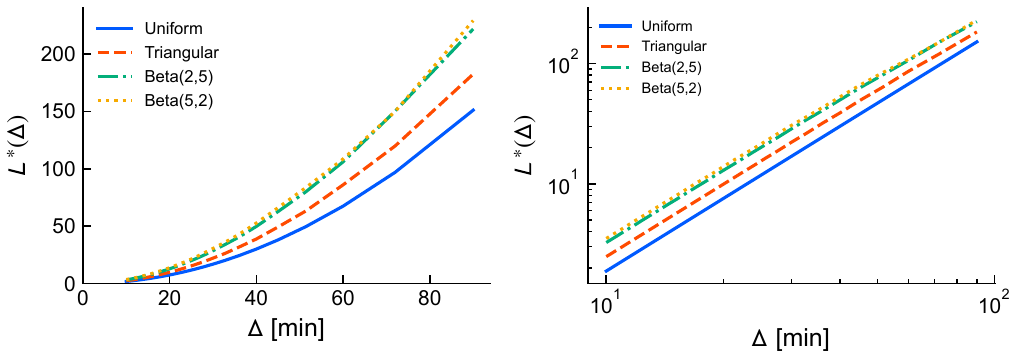}
    \subcaption{Demand prior shape.}
    \label{fig:distribution_sensitivity}
  \end{subfigure}
  \caption{Sensitivity sweeps probing the peakedness and curvature
  condition~\labelcref{eq:cc_regime} of
  \cref{thm:eps_quadratic_global}.
  Each row varies one parameter per the design of
  \cref{tab:experimental_design}, holding the others at the baseline
  of \cref{tab:baseline_parameters}.
  The left columns show $L^{\ast}(\Delta)$ and
  $\varepsilon^{\ast}(\Delta)$ on log-log axes, and the right columns
  show $\Delta^{\ast}(\omega)$ from
  \cref{eq:Delta_star}.}
  \label{fig:sensitivity_combined}
\end{figure}

The cost asymmetry sweep covers
$\ddot{c}_{\min}/\ddot{c}_{\max} \in \{1, 2/3, 1/2, 1/3, 1/5\}$,
and the slope-$2$ behavior persists at $\gamma/\beta = 5$
($\ddot{c}_{\min}/\ddot{c}_{\max} = 0.2$), well below the
sufficient threshold of \cref{eq:cc_regime}.
The prior shape sweep places
$\nu^{\mathrm{peak}}/\mu \approx 2.66$ for triangular and
$\approx 3.3$ for the Beta priors, both violating
$\nu^{\mathrm{peak}}/\mu < 2$.
The empirical log-log slopes are $2.00$ (uniform),
$1.98$ (triangular), $1.94$ (Beta(2,5)), and $1.91$ (Beta(5,2)).
The $\mathcal{O}(\Delta^{2})$ rate therefore persists well
beyond the sufficient condition~\labelcref{eq:cc_regime}, consistent with the cancellation of
\cref{eq:KKT_flow} still operating when the Taylor expansion of
the slot mean cost does not close.
A formal extension of the discrete concavity result
(\cref{lem:g_concavity}) to $\nu^{\mathrm{peak}}/\mu \ge 2$ is
left for future work.

\section{Discussion}
\label{sec:discussion}

This section interprets the main results.
\cref{subsec:rate_implications} discusses the operational
consequences of the matching $\mathcal{O}(\Delta^{2})$ rates on
strategyproofness and efficiency.
\cref{subsec:role_of_tolling} establishes that the toll is
indispensable for those rates. Without it, the strategyproofness
gap remains bounded away from zero regardless of how fine the
slot grid becomes.

\subsection{Operational Implications of the Quadratic Rates}
\label{subsec:rate_implications}

The matching $\mathcal{O}(\Delta^{2})$ rates of
\cref{thm:eps_quadratic_global,thm:eff_loss_quadratic} make
$\Delta$ the only operational parameter that the operator needs
to tune.
Halving $\Delta$ quarters both $\varepsilon^{\ast}(\Delta)$ and
$L^{\ast}(\Delta)$, so refining the slot grid improves
strategyproofness and efficiency at the same rate, with no
tradeoff between the two over this range.

The quadratic decay reaches operationally acceptable levels at
moderate $\Delta$.
At the baseline of \cref{tab:baseline_parameters},
$\Delta = 12$\,min (30 active slots spanning $\ClS$) brings the efficiency loss
to $1.0\%$ of $J^{\ast}$ (\cref{fig:efficiency_loss}) and the relative
misreport gain $\varepsilon^{\ast}(\Delta)/\bar\lambda(\Delta)$
to below $1\%$ (\cref{fig:epsilon_relative}).
The operator need not push $\Delta$ to fractions of a minute to
obtain an assignment that is close to incentive compatible and
close to efficient.

The operational slot width $\Delta^{\ast}(\omega)$ of
\cref{eq:Delta_star} converts this rate into a concrete sizing
rule.
For an indifference band $\omega$ on the relative misreport gain,
the largest admissible $\Delta$ inherits the
$\Delta^{\ast}(\omega) \propto \sqrt{\omega}$ scaling from the
$\Delta^{2}$ rate, and the robustness sweeps of
\cref{fig:sensitivity_ratio,fig:distribution_sensitivity} show
that this sizing rule survives variation in the cost asymmetry
and the demand prior shape, including parameter values outside
the peakedness and curvature condition~\labelcref{eq:cc_regime} of
\cref{thm:eps_quadratic_global}.

\subsection{Role of Tolling}
\label{subsec:role_of_tolling}
The toll~$p^{\ast}_\Delta(t)$ is essential for
the quadratic bound on $\varepsilon^{\ast}(\Delta)$.
The following proposition shows that without tolling the
misreporting gain remains bounded away from zero.

Consider the \emph{no-toll} (NT) variant in which the individual cost
is defined by the schedule delay cost alone:
\begin{align}
  C^{\mathrm{NT}}_\Delta(\hat{\theta}_{n};\, \theta_{n})
  \coloneqq \frac{1}{W}\int_{\Gamma^{\ast}_\Delta(\hat{\theta}_{n})}
  c(\theta_{n}, t)\,\mathrm{d}t,
  \label{eq:cost_NP}
\end{align}
and define the corresponding strategyproofness parameter,
consistently with \cref{eq:eps_Delta_def}, as
\begin{align}
  \varepsilon^{\ast,\mathrm{NT}}(\Delta)
  \coloneqq
  \sup_{n, \theta_{n}, \hat{\theta}_{n}}
  \bigl\{
    C^{\mathrm{NT}}_\Delta(r_{\Delta}(\theta_{n});\, \theta_{n})
    - C^{\mathrm{NT}}_\Delta(\hat{\theta}_{n};\, \theta_{n})
  \bigr\}_{+}.
  \label{eq:eps_NP_def}
\end{align}

\begin{pro}[Persistence of misreporting incentive without tolling]
  \label{prop:NP_lower_bound}
  Suppose that some vehicle~$n$ occupies a congested slot
  ($d(s_{i_{n}}) > \mu\Delta$, i.e., capacity binds at $s_{i_{n}}$)
  with preferred arrival time $\theta_{n}$ bounded away from
  $T/2$ by a fixed margin independent of $\Delta$.
  Then there exists a constant $c_{0} > 0$ depending on that
  margin but independent of~$\Delta$ such that for all
  sufficiently small $\Delta > 0$,
  \begin{align}
    \varepsilon^{\ast,\mathrm{NT}}(\Delta) \geq c_{0}.
    \label{eq:eps_NP_bound}
  \end{align}
\end{pro}

\begin{prf}
See \cref{app:proof_NP_lower_bound}.
\end{prf}

With the toll, \cref{thm:eps_quadratic_global} gives
$\varepsilon^{\ast}(\Delta) = O(\Delta^2)$.
Without it, the first-order cost variation is not priced and
\cref{prop:NP_lower_bound} gives
$\varepsilon^{\ast,\mathrm{NT}}(\Delta) \geq c_{0} > 0$ for all
sufficiently small $\Delta$.
The toll and the slot menu are therefore \emph{complementary}:
the $\ClO(\Delta^{2})$ guarantee emerges only when both are in
place (\cref{fig:role_of_tolling}).
The toll's role here is elicitation rather than
Pigovian externality internalization. The operator enforces the
assignment directly, so the toll is not needed to induce the
system optimum.
Instead, it makes truthful slot reporting incentive compatible
under coarse reports.
The efficiency loss $L^{\ast}(\Delta)$ of
\cref{thm:eff_loss_quadratic} is evaluated at the \emph{truthful}
profile. Because the toll is a transfer that cancels in the social
cost~$J$, this truthful profile loss is $\ClO(\Delta^2)$ whether or
not the toll is levied.
The toll's bearing on efficiency is therefore indirect but decisive:
without it the misreporting incentive persists
(\cref{prop:NP_lower_bound}), truthful reporting is not elicited, and
the realized assignment departs from the efficient one, so the toll
secures realized efficiency and not only incentive compatibility.

\begin{figure}[tbp]
  \centering
  \begin{tikzpicture}[
    box/.style={draw, rounded corners=3pt, minimum width=3.0cm,
                minimum height=0.9cm, align=center,
                font=\small\sffamily},
    arr/.style={->, thick, >=stealth},
    every node/.style={font=\small},
  ]
    \node[font=\bfseries\sffamily] at (0, 2.6) {(a) With tolling};

    \node[box, fill=blue!8] (delta1) at (0, 1.5)
      {Slot width $\Delta$};
    \node[box, fill=black!6] (eps1) at (-2.2, -0.3)
      {$\varepsilon^{\ast}(\Delta) = \ClO(\Delta^{2})$};
    \node[box, fill=black!6] (eff1) at ( 2.2, -0.3)
      {$L^{\ast}(\Delta) = \ClO(\Delta^{2})$\\[-1pt]{\scriptsize(realized)}};

    \draw[arr, blue!70!black] (delta1) -- (eps1)
      node[midway, left, font=\footnotesize] {vanishes};
    \draw[arr, blue!70!black] (delta1) -- (eff1)
      node[midway, right, font=\footnotesize] {vanishes};

    \node[font=\bfseries\sffamily] at (8, 2.6) {(b) Without tolling};

    \node[box, fill=red!8] (delta2) at (8, 1.5)
      {Slot width $\Delta$};
    \node[box, fill=red!10] (eps2) at (5.8, -0.3)
      {$\varepsilon^{\ast,\mathrm{NT}}(\Delta) \geq c_{0}$};
    \node[box, fill=black!6] (eff2) at (10.2, -0.3)
      {$L^{\ast}(\Delta) = \ClO(\Delta^{2})$\\[-1pt]{\scriptsize(truthful profile)}};

    \draw[arr, red!70!black] (delta2) -- (eps2)
      node[midway, left, font=\footnotesize] {persists};
    \draw[arr, blue!70!black] (delta2) -- (eff2)
      node[midway, right, font=\footnotesize] {vanishes};

    \draw[gray, dashed] (4, -1.0) -- (4, 3.0);
  \end{tikzpicture}
  \caption{Role of the toll. With tolling (a), the misreporting
  incentive vanishes ($\varepsilon^{\ast} = \ClO(\Delta^2)$) and the
  efficient assignment is realized. Without tolling (b), it persists
  ($\varepsilon^{\ast,\mathrm{NT}} \geq c_0$), so truthful reporting is
  not elicited, even though the truthful profile loss remains
  $\ClO(\Delta^2)$. The toll is thus essential for realized efficiency.}
  \label{fig:role_of_tolling}
\end{figure}
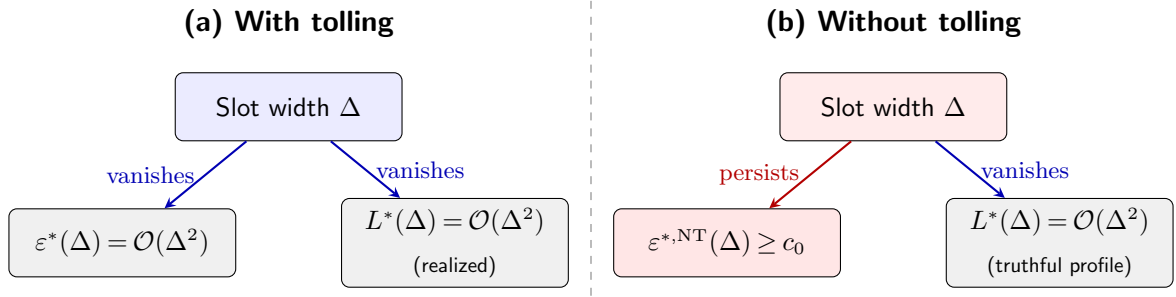

As $\Delta \to 0$, the managed lane toll
$p^{\ast}_{\Delta}$
defined by~\cref{eq:toll_expression} converges pointwise to the
classical continuous-report congestion toll
\citep{Vickrey1969-ic,ArnottDePalmaLindsey1994-wc,Small2015-lc}
via Lipschitz continuity of $\lambda^{\ast}_{\Delta}$ in $s$ and
$c$ in $\theta$ together with $r_{\Delta}(\theta) \to \theta$.

\cref{fig:tolling_multiload} supports these bounds numerically
across three load levels
$N/(S\mu) \in \{1.11, 1.33, 1.67\}$ obtained by varying $N$ at
fixed $\mu$ and $S$.
With tolling, $\varepsilon^{\ast}(\Delta)$ lies below the
$H_{\varepsilon}\Delta^{2}$ envelope at every load.
Without tolling, $\varepsilon^{\ast,\mathrm{NT}}(\Delta)$ tracks
the theoretical floor $c_{0} = \beta D_{n}^{2}/2$, which scales as
$(N/(S\mu) - 1)^{2}$ and grows from $\approx 0.033$ at
$N/(S\mu) = 1.11$ to $\approx 1.2$ at $N/(S\mu) = 1.67$.
The strict separation between the tolled and untolled cases is therefore
robust to the choice of load.

\begin{figure}[tbp]
  \centering
  \includegraphics[width=0.95\textwidth]{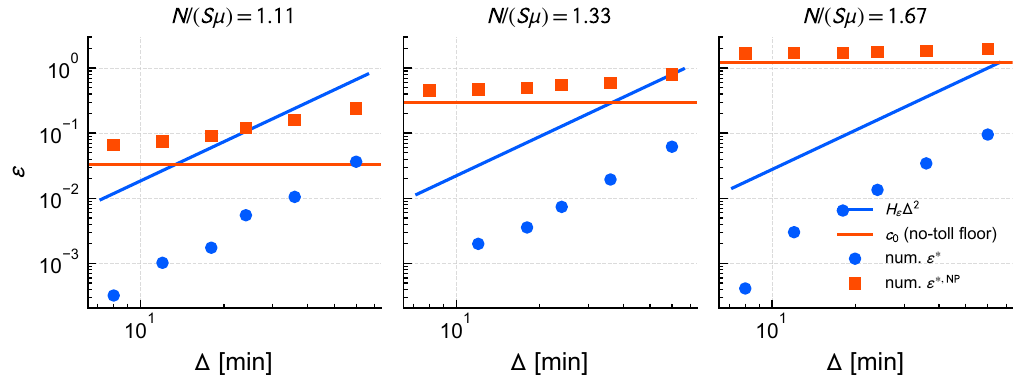}
  \caption{Tolling vs.\ no-tolling at three load levels
  $N/(S\mu) \in \{1.11, 1.33, 1.67\}$.
  Blue curves and markers show the tolled mechanism with the
  $H_{\varepsilon}\Delta^{2}$ envelope, and red curves and markers
  show the no-tolling variant with the
  $c_{0}$ floor.
  The strict separation between the tolled and untolled cases holds at every
  load.}
  \label{fig:tolling_multiload}
\end{figure}

\section{Conclusion}
\label{sec:conclusion}

We have analyzed the slot-based reporting mechanism for
system-optimal bottleneck assignment on a
managed lane,
characterizing the worst-case strategyproofness parameter
$\varepsilon^{\ast}(\Delta)$ and the expected efficiency loss
$L^{\ast}(\Delta)$ as functions of the slot width $\Delta$.
The main results are
$\varepsilon^{\ast}(\Delta) \leq C_{\varepsilon}\,\Delta^{2}$
for every $\Delta \in (0, \Delta_{\max}]$ under the peakedness and
curvature condition~\labelcref{eq:cc_regime} (\cref{thm:eps_quadratic_global}) and
$L^{\ast}(\Delta) \leq N\,C_{L}\,\Delta^{2}$ for every
$\Delta > 0$ under binding capacity
(\cref{thm:eff_loss_quadratic}), with no curvature condition required.
Without the toll, the misreporting incentive remains
$\mathcal{O}(1)$ regardless of slot refinement
(\cref{prop:NP_lower_bound}).

The framework extends to other congested systems with finite
capacity and scheduling preferences, including airport runway
scheduling and port berth allocation.
Two natural extensions are heterogeneous users (multiple
value-of-time or schedule-delay classes
\citep{ArnottDePalmaLindsey1994-wc, Lindsey2004-mu,
vandenBergVerhoef2011-wd, Akamatsu2021-di}) and corridor
networks with tandem bottlenecks
\citep{Fu2022-st, Sakai2024-wy}. Both require generalizations of
the sorting structure and the cross-bottleneck propagation
respectively.

\section*{Acknowledgements}
This work was supported by JSPS KAKENHI Grant Numbers
JP26K17486 and JP25H00751.

\section*{Declaration of generative AI use}
During the preparation of this work the authors used the Claude
and GPT model families for coding and writing assistance.
After using these tools, the authors reviewed and edited
the content as needed and take full responsibility for the
content of the published article.

\newpage
\begin{APPENDIX}{}

\section{Notation Summary}
\label{app:notation}

\cref{tab:notation} lists the main symbols used throughout this
paper.

\begin{table}[tbp]
  \centering
  \caption{Summary of notation.}
  \label{tab:notation}
\begin{tabular}{cl}
  \toprule
  Symbol & Description \\
  \midrule
  \multicolumn{2}{l}{\textit{Model primitives (exogenous)}} \\
  \addlinespace[0.2em]
  $\ClN = \{1, \ldots, N\}$ & Set of users (vehicles) \\
  $\ClT = [0, T]$ & Assignment window (where actual arrivals are scheduled) \\
  $T$ & Length of the assignment window $\ClT$ ($\Delta \mid T$) \\
  $\ClS = [\theta^{-}, \theta^{+}] \subseteq \ClT$ & Support of preferred arrival times \\
  $S = \theta^{+} - \theta^{-}$ & Width of the support $\ClS$ \\
  $\mu$ & Bottleneck capacity \\
  $f$ & Free-flow travel time \\
  $\nu(\theta)$ & Density of preferred arrival times on $\ClS$ (\cref{asm:demand_density}) \\
  $c(\theta, t)$ & Schedule delay cost function \\
  $\dot{c}_{\max}$ & Lipschitz constant of $c$: $\sup |\partial c/\partial\theta|$ (\cref{asm:schedule_cost}(C3)) \\
  $\ddot{c}_{\max},\; \ddot{c}_{\min}$ & Upper/lower bounds on $\partial^2 c/\partial\theta^2$ (\cref{asm:schedule_cost}(C4)) \\
  \midrule
  \multicolumn{2}{l}{\textit{Reports and slot structure}} \\
  \addlinespace[0.2em]
  $\theta_{n} \in \ClS$ & True preferred arrival time of vehicle $n$ (private information) \\
  $\Vttheta = (\theta_{1}, \ldots, \theta_{N})$ & Profile of true preferred arrival times \\
  $\hat{\theta}_{n} \in \ClS_\Delta$ & Reported preferred arrival time of vehicle $n$ \\
  $\Delta$ & Slot width (a divisor of $T$) \\
  $\ClS_\Delta = \{s_{1}, \ldots, s_{M}\}$ & Set of admissible slot reports (covering $\ClT$) \\
  $M = T/\Delta$ & Number of slots \\
  $r_\Delta(\theta)$ & Mapping from $\theta$ to its nearest slot (\cref{subsec:mechanism}) \\
  \midrule
  \multicolumn{2}{l}{\textit{Demand structure (derived from reports)}} \\
  \addlinespace[0.2em]
  $d(s_{i})$ & Vehicles reporting slot $s_{i}$ under profile $\hat\Vttheta$ (\cref{eq:demand}) \\
  $\nu^{\mathrm{peak}} = \max_{\theta \in \ClS} \nu(\theta)$ & Peak arrival rate of $\nu$ (\cref{eq:nu_peak_def}) \\
  \midrule
  \multicolumn{2}{l}{\textit{Mechanism solution ($\Delta$-DSO)}} \\
  \addlinespace[0.2em]
  $q_\Delta(s, t)$ & Flow rate assigned to arrival time $t$ among vehicles reporting slot $s$ \\
  $\Gamma^{\ast}_\Delta(s)$ & Arrival-time interval assigned to slot $s$ \\
  $W_\Delta(s_{i}) = d(s_{i})/\mu$ & Assignment interval width for slot $s_{i}$ \\
  $\lambda^{\ast}_\Delta(s)$ & Equilibrium cost for slot $s$ (multiplier of demand-conservation constraint) \\
  $p^{\ast}_\Delta(t)$ & Time-dependent managed lane toll \\
  $L_{p}$ & Lipschitz constant of the toll $p^{\ast}_\Delta(t)$ \\
  \midrule
  \multicolumn{2}{l}{\textit{Cost evaluation and performance metrics}} \\
  \addlinespace[0.2em]
  $C_\Delta(\hat{\theta}_{n};\, \theta_{n})$ & Ex-ante expected cost of a price-taking user (within-slot mean, \cref{eq:individual_cost}) \\
  $\varepsilon^{\ast}(\Delta)$ & Tightest $\varepsilon$-strategyproofness parameter \\
  $L^{\ast}(\Delta)$ & Expected efficiency loss (\cref{eq:eff_loss_def}) \\
  \bottomrule
\end{tabular}
\end{table}

\section{Proof Details}
\label{app:proofs}

Throughout this appendix we use the slot occupancy
$d(s) = \#\{n : \hat{\theta}_{n} = s\}$ under the report
profile, equal to $\int_{r_\Delta^{-1}(s)} \nu(\theta)\,\mathrm{d}\theta$
under truthful reporting, and the corresponding assignment
interval width
\begin{align}
  W_\Delta(s) \coloneqq \frac{d(s)}{\mu}.
  \label{eq:Wdelta_def}
\end{align}
Both depend on the slot width $\Delta$ and on the reported
preferred arrival time distribution through $d(s)$. The
peak truthful occupancy satisfies
$\max_{s} d(s)/(\mu\Delta) \leq \nu^{\mathrm{peak}}/\mu$.
We also use the shorthand
\begin{align}
  \phi(z) \coloneqq \frac{(z-1)(2z-1)}{z^{2}}
  \label{eq:phi_def}
\end{align}
for the curvature-ratio threshold of \cref{eq:cc_regime}.

\subsection{Proof of \cref{lem:FIFW}}
\label{app:proof_FIFW}

\begin{prf}
  Suppose, for the sake of contradiction, that there exist two
  pairs $(s, t)$ and $(s', t')$ with $s < s'$ and $t > t'$
  (i.e., a violation of the sorting property) such that
  $q_{\Delta}(s,t) > 0$ and $q_{\Delta}(s',t') > 0$.
  Define
  $\delta \coloneqq \min\{q_{\Delta}(s,t),\, q_{\Delta}(s',t')\} > 0$
  and reassign $\delta$ units of flow
  from $(s,t)$ and $(s',t')$ to $(s,t')$ and $(s',t)$, respectively.
  The resulting change in the objective function is
  \begin{align}
    \delta\bigl[c(s, t') + c(s', t) - c(s, t) - c(s', t')\bigr].
  \end{align}
  By the submodularity condition \labelcref{eq:submodular}, because
  $s < s'$ and $t' < t$ we have
  $c(s, t') + c(s', t) \leq c(s, t) + c(s', t')$,
  so the above change is non-positive.
  That is, the reassignment reduces (or leaves unchanged) the
  objective function value.
  Hence any assignment that violates the sorting property is not optimal,
  and the optimal solution obeys the sorting property.
\end{prf}

\subsection{Proof of \cref{lem:toll_Lip}}
\label{app:proof_toll_Lip}

\begin{prf}
  From \cref{eq:toll_expression},
  $p^{\ast}_{\Delta}(t) = \lambda^{\ast}_{\Delta}(\slotmap(t)) - c(\slotmap(t), t)$.

  \textit{Interior of each interval:}
  Under the sorting property, within each assignment interval
  $\Gamma^{\ast}_{\Delta}(s) = [t^{-}_\Delta(s), t^{+}_\Delta(s)]$
  we have $\slotmap(t) = s$ (the assigned slot is constant), so
  for any $t, t' \in \Gamma^{\ast}_\Delta(s)$,
  \begin{align}
    |p^{\ast}_\Delta(t) - p^{\ast}_\Delta(t')| = |c(s, t') - c(s, t)|.
  \end{align}
  By condition (C1) of \cref{asm:schedule_cost}, the schedule delay cost
  $c(s, \cdot)$ is continuously differentiable, and on the bounded
  domain $\ClT$ there exists a Lipschitz constant $L_{p}$ with respect
  to $t$, giving
  $|p^{\ast}_\Delta(t) - p^{\ast}_\Delta(t')| \leq L_{p} |t - t'|$.

  \textit{Equilibrium cost difference:}
  At a boundary time
  $t_{b} = t^{+}_\Delta(s_{i}) = t^{-}_\Delta(s_{i+1})$,
  the optimality condition~\labelcref{eq:KKT_flow} holds simultaneously
  for $s_{i}$ and $s_{i+1}$, giving
  $\lambda^{\ast}_\Delta(s_{i+1}) - \lambda^{\ast}_\Delta(s_{i})
  = c(s_{i+1}, t_{b}) - c(s_{i}, t_{b})$.
  By~(C3), $|\lambda^{\ast}_\Delta(s_{i+1}) - \lambda^{\ast}_\Delta(s_{i})|
  \leq \dot{c}_{\max} \Delta$.
  Recursive application yields
  $|\lambda^{\ast}_\Delta(s_{j}) - \lambda^{\ast}_\Delta(s_{i})|
  \leq \dot{c}_{\max} |s_{j} - s_{i}|$ for all $i, j$.

  \textit{Continuity at interval boundaries:}
  The identity above implies
  $p^{\ast}_\Delta(t_{b})
    = \lambda^{\ast}_\Delta(s_{i}) - c(s_{i}, t_{b})
    = \lambda^{\ast}_\Delta(s_{i+1}) - c(s_{i+1}, t_{b})$,
  so the toll is continuous at each boundary.
  Combining the Lipschitz continuity within each interval with
  continuity at the boundaries, Inequality~\labelcref{eq:toll_Lip} holds globally
  with $L_{p} = \dot{c}_{\max}$.
\end{prf}

\subsection{Proof of \cref{lem:within_slot_mean}}
\label{app:proof_within_slot_mean}

\begin{prf}
  Let $f(t) \coloneqq c(\theta, t) + p^{\ast}_\Delta(t)$ on
  $\Gamma^{\ast}_\Delta(s) = [a, a + W_\Delta(s)]$ with
  $W_\Delta(s) = d(s)/\mu$.
  Within $\Gamma^{\ast}_\Delta(s)$, both $c(\theta, \cdot)$ and
  $p^{\ast}_\Delta$ are piecewise $C^2$ in $t$:
  $|c_{tt}| \leq \ddot{c}_{\max}$ by~(C1) and~(C4) of
  \cref{asm:schedule_cost}, and
  $|(p^{\ast}_\Delta)''| \leq \ddot{c}_{\max}$ by the toll
  formula~\labelcref{eq:toll_expression} (the toll inherits the second
  derivative bound from $c$).
  Hence $|f''| \leq 2 \ddot{c}_{\max}$ on the interior of
  $\Gamma^{\ast}_\Delta(s)$, with possible isolated kinks at the
  $C^2$-partition boundaries of \cref{asm:schedule_cost}(C4).

  Partition $\Gamma^{\ast}_\Delta(s)$ into $d(s)$ subintervals
  $I_{k} = [a + (k-1)/\mu,\, a + k/\mu]$ of length $1/\mu$ each,
  with midpoints $\tau_{k} = a + (k - 1/2)/\mu$.
  By the midpoint quadrature error formula on each $I_{k}$,
  \begin{align}
    \int_{I_{k}} f(t)\,\mathrm{d}t
    = \frac{f(\tau_{k})}{\mu} + \frac{f''(\xi_{k})}{24\mu^3},
    \qquad \xi_{k} \in I_{k},
  \end{align}
  for each $C^2$ subinterval.
  Summing over $k = 1, \ldots, d(s)$ and dividing by $W_\Delta(s) = d(s)/\mu$,
  \begin{align}
    \frac{1}{W_\Delta(s)} \int_{\Gamma^{\ast}_\Delta(s)} f(t)\,\mathrm{d}t
    = \frac{1}{d(s)} \sum_{k=1}^{d(s)} f(\tau_{k})
      + \frac{1}{d(s)} \sum_{k=1}^{d(s)}
        \frac{f''(\xi_{k})}{24\mu^2}.
  \end{align}
  The remainder is bounded by
  \begin{align}
    \left|\frac{1}{d(s)} \sum_{k=1}^{d(s)} \frac{f''(\xi_{k})}{24\mu^2}\right|
    \leq \frac{2 \ddot{c}_{\max}}{24\mu^2}
    = \frac{\ddot{c}_{\max}}{12\mu^2}.
  \end{align}
  When the slot contains at least one vehicle, $1 \leq d(s) \leq
  \nu^{\mathrm{peak}}\Delta$, so $1/\mu \leq (\nu^{\mathrm{peak}}/\mu)\Delta$
  and the remainder is
  $\mathcal{O}\bigl((\nu^{\mathrm{peak}}/\mu)^{2}\ddot{c}_{\max}\Delta^{2}\bigr)
  = \mathcal{O}(\Delta^{2})$ uniformly in $s, \theta$.
  This discrete reinterpretation requires $d(s) \geq 1$, i.e.,
  $\Delta \geq 1/\nu^{\mathrm{peak}}$, and the bound $\ddot{c}_{\max}/(12\mu^2)$
  is otherwise a $\Delta$-independent floor.
  In the nonatomic (fluid) model used for the main results the
  decision cost is the continuous within-slot mean $C_\Delta$ of
  \cref{eq:individual_cost}, whose deviation from the centroid value
  is the genuine $\ClO(\ddot{c}_{\max}\Delta^{2})$ within-slot variance
  term. The $\ClO(1/\mu^{2})$ floor here is a finite-population
  artifact of the discrete $1/\mu$-spaced placement.
  Kinks of $c$ at the $C^2$-partition boundaries contribute at
  most a $\Delta$-independent finite number (the partition
  cardinality given by~(C4) of \cref{asm:schedule_cost}) of
  additional $\mathcal{O}(1/\mu^2) = \mathcal{O}(\Delta^2)$
  corrections, which are absorbed into the same residual.
\end{prf}

\subsection{Proof of \cref{lem:eps_quadratic}}
\label{app:proof_eps_quadratic}

\begin{prf}
We bound the adjacent-slot gain
$g_{\Delta}(+1) = C_\Delta(s_{i};\, \theta_{n})
- C_\Delta(s_{i+1};\, \theta_{n})$ in two steps:
(i) bound the slotwise gradient of the generalized cost
$G = c + p^{\ast}_\Delta$, using that the toll cancels the
schedule delay cost gradient on each slot's own interval, and
(ii) integrate this gradient bound over the adjacent intervals
to obtain $|g_{\Delta}(+1)| = \ClO(\Delta^{2})$.
The $-1$ direction is symmetric.
Here $s_{i} = r_\Delta(\theta_{n})$ and
$|\theta_{n} - s_{i}| \leq \Delta/2$.

\textit{Step 1: Slotwise gradient bound.}
Define $G(\theta, t) \coloneqq c(\theta, t) + p^{\ast}_\Delta(t)$.
Within each assignment interval $\Gamma^{\ast}_\Delta(s)$,
the optimality condition~\labelcref{eq:KKT_flow} gives
$p^{\ast}_\Delta(t) = \lambda^{\ast}_\Delta(s) - c(s, t)$, so
\begin{align}
  G_{t}(\theta, t)
  = c_{t}(\theta, t) - c_{t}(s, t),
  \qquad t \in \Gamma^{\ast}_\Delta(s).
  \label{eq:Gt}
\end{align}
At $\theta = s$, this vanishes identically because the toll cancels the
schedule delay cost gradient on the slot's own interval.
By the piecewise $C^2$ regularity ((C4) of \cref{asm:schedule_cost}) and
$|c_{tt}| = |f''| \leq \ddot{c}_{\max}$, for any $\theta$ and any $s$,
\begin{align}
  |G_{t}(\theta, t)|
  \;\leq\; \ddot{c}_{\max}\,|\theta - s|,
  \qquad t \in \Gamma^{\ast}_\Delta(s).
  \label{eq:Gt_bound}
\end{align}
Writing $\delta \coloneqq \theta_{n} - s_{i}$ with $|\delta| \leq \Delta/2$,
this yields two slot-dependent bounds:
$|G_{t}(\theta_{n}, t)| \leq \ddot{c}_{\max}|\delta| \leq \ddot{c}_{\max}\Delta/2$
for $t \in \Gamma_{i}$, and
$|G_{t}(\theta_{n}, t)| \leq \ddot{c}_{\max}|\delta - \Delta|
  \leq \tfrac{3}{2} \ddot{c}_{\max}\Delta$
for $t \in \Gamma_{i+1}$.

\textit{Step 2: Adjacent-interval mean difference.}
The individual cost is the integral mean of~$G$ over the
assignment interval:
$C_\Delta(s;\, \theta) = \frac{1}{W}\int_{\Gamma^{\ast}_\Delta(s)}
G(\theta, t)\,\mathrm{d}t$, where $W = W_\Delta(s)$.
Let $\Gamma_{i} = [a_{i}, a_{i} + W]$ and
$\Gamma_{i+1} = [a_{i} + W, a_{i} + 2W]$ be two adjacent intervals.
Changing variables in the second integral gives
\begin{align}
  g_{\Delta}(+1) \;=\;
  \frac{1}{W}\int_{0}^W
  \bigl[G(\theta_{n}, a_{i} {+} u)
  - G(\theta_{n}, a_{i} {+} W {+} u)\bigr]\,\mathrm{d}u.
  \label{eq:mean_diff}
\end{align}
For each fixed $u \in [0, W]$, the two sample points $a_{i} + u$ and
$a_{i} + W + u$ straddle the slot boundary $a_{i} + W$, with the first lying
in $\Gamma_{i}$ and the second in $\Gamma_{i+1}$.
Applying the fundamental theorem of calculus
along the segment $[a_{i} + u,\, a_{i} + W + u]$ and using the
slotwise gradient bounds~\labelcref{eq:Gt_bound},
\begin{align}
  &\bigl|G(\theta_{n}, a_{i} {+} u) - G(\theta_{n}, a_{i} {+} W {+} u)\bigr|
  \notag\\
  &\quad\leq
  \int_{a_{i} + u}^{a_{i} + W} |G_{t}(\theta_{n}, \tau)|\,\mathrm{d}\tau
  + \int_{a_{i} + W}^{a_{i} + W + u} |G_{t}(\theta_{n}, \tau)|\,\mathrm{d}\tau
  \notag\\
  &\quad\leq
  (W - u)\cdot\frac{\ddot{c}_{\max}\Delta}{2}
  + u\cdot\frac{3 \ddot{c}_{\max}\Delta}{2}
  \;=\; \ddot{c}_{\max}\Delta\Bigl(\frac{W}{2} + u\Bigr).
  \label{eq:pointwise_diff}
\end{align}
Integrating over $u \in [0, W]$ and dividing by $W$,
\begin{align}
  |g_{\Delta}(+1)|
  \;\leq\;
  \frac{1}{W}\int_{0}^W \ddot{c}_{\max}\Delta\Bigl(\frac{W}{2} + u\Bigr)\,\mathrm{d}u
  \;=\; \ddot{c}_{\max}\Delta\cdot W
  \;\leq\; \ddot{c}_{\max}\,(\nu^{\mathrm{peak}}/\mu)\,\Delta^2,
  \label{eq:g_bound_quadratic}
\end{align}
which is the stated bound.
Since $\varepsilon^{\ast}(\Delta)$ in \cref{eq:eps_Delta_def}
takes the positive part $\{\cdot\}_{+}$, the bound on
$|g_{\Delta}(+1)|$ suffices.
The $-1$ direction is symmetric.
This matches the bound stated in \cref{lem:eps_quadratic},
since $W = W_\Delta(s) \leq \nu^{\mathrm{peak}}\Delta/\mu$ on any
active slot.
\end{prf}

The absolute-value step in the bound above ignores sign
cancellations across the slot boundary.
A direct computation for \cref{eq:quadratic_cost} with
$\theta_{n}$ at the slot center gives
$|g_{\Delta}(\pm 1)| = \ddot{c}_{\max}\,(\nu^{\mathrm{peak}}/\mu)\,\Delta^2/2$,
so the true worst-case constant lies between
$\ddot{c}_{\max}\,(\nu^{\mathrm{peak}}/\mu)/2$ and
$\ddot{c}_{\max}\,(\nu^{\mathrm{peak}}/\mu)$.
We keep the looser $\ddot{c}_{\max}\,(\nu^{\mathrm{peak}}/\mu)$ because the
slot dependent gradient argument extends cleanly to the
multi-slot proof of \cref{thm:eps_quadratic_global}.
This factor-of-two looseness is consistent with the empirical
prefactor of \cref{sec:numerical} lying well below the
theoretical bound.

\subsection{Proof of \cref{thm:eps_quadratic_global}}
\label{app:proof_epsilon}

We first record a smoothness bound on the assignment interval width
(\cref{app:proof_width_smoothness}). We then prove the theorem in two
steps.
First, we establish a strict discrete concavity of the gain
sequence $g_{\Delta}(k)$ (\cref{app:g_concavity}), which bounds
the maximizer $k^{\ast}$ by a $\Delta$-independent constant.
Second, we combine this concavity with a slot-level quadratic
bound on each marginal gain $D(j)$ to derive the global
$\mathcal{O}(\Delta^{2})$ bound (\cref{app:proof_epsilon_main}).

\subsubsection{Smoothness of the assignment interval width}
\label{app:proof_width_smoothness}

\begin{lem}[Smoothness of the assignment interval width]
  \label{lem:width_smoothness}
  Under \cref{asm:demand_density}, the assignment interval widths
  $W_\Delta(s) = d(s)/\mu$ satisfy
  \begin{align}
    |W_\Delta(s_{i+1}) - W_\Delta(s_{i})|
    \leq \frac{L_\nu \Delta^2}{\mu}
    \label{eq:width_smoothness}
  \end{align}
  for any pair of adjacent active slots.
  Consequently, the truthful profile peak occupancy
  $\max_{i} d(s_{i})/(\mu\Delta)$ deviates from the rate
  ratio $\nu^{\mathrm{peak}}/\mu$ by at most $\mathcal{O}(L_\nu)$
  uniformly in the slot index, and converges to
  $\nu^{\mathrm{peak}}/\mu$ as $\Delta \to 0$.
\end{lem}

\begin{prf}
By \cref{asm:demand_density}, the truthful occupancy is
$d(s_{i}) = \int_{[s_{i} - \Delta/2,\, s_{i} + \Delta/2)}
\nu\,\mathrm{d}\theta$, so
$|d(s_{i+1}) - d(s_{i})| \leq L_\nu \Delta^2$
by the Lipschitz continuity of $\nu$
(\cref{asm:demand_density}(D2)) and translation invariance of the
integration window.
Dividing by $\mu$ yields~Inequality~\labelcref{eq:width_smoothness}.
Since adjacent occupancies differ by $\mathcal{O}(L_\nu\Delta^2)$
and $\max_{i} d(s_{i}) \leq \nu^{\mathrm{peak}}\Delta$, the
peak ratio $\max_{i} d(s_{i})/(\mu\Delta)$ deviates from
$\nu^{\mathrm{peak}}/\mu$ by at most $\mathcal{O}(L_\nu)$ and
converges to it as $\Delta \to 0$.
\end{prf}

\subsubsection{Strict discrete concavity of the gain sequence}
\label{app:g_concavity}

The strict concavity uses the piecewise $C^2$ structure of both
the schedule delay cost and the toll.
A slot-mean Taylor expansion enabled by \cref{asm:demand_density}
then yields the following strict discrete concavity.

\begin{lem}[Strict discrete concavity of the gain sequence]
  \label{lem:g_concavity}
  Assume \cref{asm:schedule_cost,asm:demand_density}, the
  peakedness condition
  $\nu^{\mathrm{peak}}/\mu < 2$, and the curvature
  condition
  \begin{align}
    \frac{\ddot{c}_{\min}}{\ddot{c}_{\max}}
    - \phi\!\left(\frac{\nu^{\mathrm{peak}}}{\mu}\right) > 0.
    \label{eq:kappa_g_positivity}
  \end{align}
  Define the slotwise concavity coefficient
  \begin{align}
    \kappa_{g}(s) \coloneqq
    \ddot{c}_{\min}\,\bigl(W_\Delta(s)\bigr)^{2}
    - \ddot{c}_{\max}\,\bigl(W_\Delta(s) - \Delta\bigr)
      \bigl(2W_\Delta(s) - \Delta\bigr)
    \label{eq:kappa_g_def}
  \end{align}
  and its global lower bound
  $\bar\kappa_{g} \coloneqq \min_{i:\, d(s_{i}) > 0}
  \kappa_{g}(s_{i}) > 0$.
  For every interior step $k \geq 1$ such that
  $s^{k-1}, s^k, s^{k+1}$ are active,
  \begin{align}
    D(k) - D(k{-}1)
    \leq -\kappa_{g}(s^k) + R_\nu(\Delta),
    \label{eq:D_second_diff}
  \end{align}
  where $R_\nu(\Delta) = (\dot{c}_{\max}\,L_\nu/\mu)\,\Delta^2
  + \mathcal{O}(\Delta^3)$ is the smoothness correction induced
  by \cref{asm:demand_density} and \cref{lem:width_smoothness}.
\end{lem}

\begin{cor}[Uniform discrete concavity]
  \label{cor:g_concavity_uniform}
  Under the assumptions of \cref{lem:g_concavity} and the
  demand smoothness threshold
  \begin{align}
    L_\nu \;<\; \frac{\mu\,\bar\kappa_{g}}{\dot{c}_{\max}\,\Delta^2},
    \label{eq:smoothness_threshold}
  \end{align}
  the bound \labelcref{eq:D_second_diff} reduces to
  \begin{align}
    D(k) - D(k{-}1)
    \leq -\bigl(\bar\kappa_{g} - \dot{c}_{\max}\,L_\nu\Delta^2/\mu\bigr)
    + \mathcal{O}(\Delta^3),
  \end{align}
  uniformly in $k$, with a strictly positive effective concavity
  margin $\bar\kappa_{g} - \dot{c}_{\max}\,L_\nu\Delta^2/\mu > 0$.
\end{cor}

The peakedness condition $\nu^{\mathrm{peak}}/\mu < 2$, together with the
modelling scope $\nu^{\mathrm{peak}}/\mu > 1$ of \cref{subsec:FIFW},
places the peak load in $(1, 2)$.
The bound $\bar\kappa_{g} > 0$ follows from
Inequality~\labelcref{eq:kappa_g_positivity}. Factoring
$\kappa_{g}(s) = \ddot{c}_{\max}\,W_\Delta(s)^2
[\ddot{c}_{\min}/\ddot{c}_{\max} - \phi(W_\Delta(s)/\Delta)]$, and
since $\phi$ is increasing with
$W_\Delta(s)/\Delta \leq \nu^{\mathrm{peak}}/\mu$
on active slots (a single price-taking user adds no mass), the
bracket is bounded below by
$\ddot{c}_{\min}/\ddot{c}_{\max} - \phi(\nu^{\mathrm{peak}}/\mu) > 0$
uniformly over active slots. As $W_\Delta(s)^2 > 0$, this gives
$\kappa_{g}(s) > 0$ for every active slot and hence
$\bar\kappa_{g} > 0$, without identifying the minimizing slot.
\begin{prf}
  The proof bounds the second difference of the slot-mean cost
  $C_\Delta(s^j;\theta_{n})$, which is the negative of
  $D(k) - D(k{-}1)$ by uniform slot spacing.
  We (i) decompose $C_\Delta$ into a schedule delay cost term and
  a toll term, (ii) bound the second difference of the schedule
  delay cost terms via a Taylor expansion, (iii) bound the second
  difference of the toll term via continuity at slot boundaries,
  and (iv) combine the three contributions to obtain
  $\kappa_g(s^k)$.

  Slot midpoints $s^{k-1}, s^k, s^{k+1}$ are uniformly spaced by
  $\Delta$, so
  \begin{align}
    D(k) - D(k{-}1)
    = -\bigl[C_\Delta(s^{k+1};\theta_{n}) - 2 C_\Delta(s^k;\theta_{n})
        + C_\Delta(s^{k-1};\theta_{n})\bigr].
    \label{eq:D2C}
  \end{align}

  \textit{Step 1: Decomposition of $C_\Delta$.}
  By the toll formula
  $p^{\ast}_\Delta(t) = \lambda^{\ast}_\Delta(s^j) - f(s^j - t)$
  on $\Gamma^{\ast}_\Delta(s^j) = [a^j, a^j + W_{j}]$ with
  $W_{j} = W_\Delta(s^j) = d(s_{j})/\mu$, changing variables in
  $C_\Delta(s^j;\theta_{n}) =
   \tfrac{1}{W_{j}}\int_{\Gamma^{\ast}_\Delta(s^j)}
  [f(\theta_{n} - t) + p^{\ast}_\Delta(t)]\,\mathrm{d}t$ gives
  \begin{align}
    C_\Delta(s^j;\theta_{n})
    =
    \underbrace{A_{j}}_{\substack{\text{schedule delay cost}\\\text{at }\theta_n\text{ averaged}\\\text{over slot }s^j}}
    +
    \underbrace{\lambda^{\ast}_\Delta(s^j)}_{\substack{\text{equilibrium}\\\text{cost of slot }s^j}}
    -
    \underbrace{B_{j}}_{\substack{\text{schedule delay cost}\\\text{at the slot center }s^j\\\text{averaged over the same interval}}},
    \label{eq:C_decomp}
  \end{align}
  where, with $F_{W}(z) \coloneqq \tfrac{1}{W}\int_{z-W}^{z} f(u)\,\mathrm{d}u$
  and $\alpha_{j} \coloneqq s^j - a^j$,
  \begin{align}
    A_{j} = F_{W_{j}}(\theta_{n} - a^j),
    \qquad
    B_{j} = F_{W_{j}}(\alpha_{j}).
  \end{align}
  We take slot $k$ as the local reference. Write
  $W_{j} = W_{k} + (W_{j} - W_{k})$ and decompose
  $A_{j} = F_{W_{k}}(\theta_{n} - a^j) + (F_{W_{j}} - F_{W_{k}})(\theta_{n} - a^j)$,
  similarly for $B_{j}$.
  By~\cref{asm:demand_density},
  $|W_{j} - W_{k}| \leq |j - k| L_\nu \Delta^2/\mu$, so all
  reference-frame corrections below ($A$, $B$, and
  $\lambda^{\ast}_{\Delta}$) are uniformly
  $\mathcal{O}(\dot{c}_{\max} L_\nu\Delta^2/\mu)$ for
  $|j - k| \leq 1$ and absorbed into the single smoothness
  correction $R_\nu(\Delta)$ defined in~\cref{eq:Rnu_def} below.
  Because the assignment intervals tile, $a^{j+1} = a^j + W_{j}$, so
  \begin{align}
    \theta_{n} - a^{j+1} = (\theta_{n} - a^j) - W_{j},
    \quad
    \alpha_{j+1} = \alpha_{j} - (W_{j} - \Delta).
    \label{eq:uniform_steps}
  \end{align}
  Within the reference frame ($W_{j} \to W_{k}$), the arguments of
  $F_{W_{k}}$ in $A$ and $B$ traverse points spaced by $-W_{k}$ and
  $-(W_{k} - \Delta)$, respectively, plus $\mathcal{O}(\Delta^2)$
  smoothness deviations.

  \textit{Step 2: Schedule delay cost contribution.}
  Differentiating $F_{W_{k}}$ gives
  $F_{W_{k}}'(z) = (1/W_{k})[f(z) - f(z - W_{k})]$, with
  $|F_{W_{k}}'| \leq \dot{c}_{\max}$ by Lipschitz continuity of $f$, and
  \begin{align}
    F_{W_{k}}''(z) = \frac{1}{W_{k}}\int_{z-W_{k}}^{z} f''(u)\,\mathrm{d}u
    \in [\ddot{c}_{\min}, \ddot{c}_{\max}]
    \label{eq:Fpp_bound}
  \end{align}
  by~(C4) of \cref{asm:schedule_cost}.
  For any $C^2$ function $\Phi$ with $\Phi'' \in [m, M]$ and points
  $a, b, c$ with $b - a = h_{1}$ and $c - b = h_{2}$, Taylor expansion
  gives
  \begin{align}
    \Phi(c) - 2\Phi(b) + \Phi(a)
    = \Phi'(b)(h_{2} - h_{1})
    + \tfrac{1}{2}\Phi''(\xi_{1})\,h_{2}^2
    + \tfrac{1}{2}\Phi''(\xi_{2})\,h_{1}^2,
    \label{eq:Taylor_nonuniform}
  \end{align}
  where $\xi_{1} \in (b, c)$, $\xi_{2} \in (a, b)$.
  Apply~\cref{eq:Taylor_nonuniform} to $A$, with the points
  $a = z_{k-1} = z_{k} + W_{k-1}$, $b = z_{k}$, $c = z_{k+1} = z_{k} - W_{k}$:
  the steps are $h_{1} = -W_{k-1}$ and $h_{2} = -W_{k}$, so
  $h_{2} - h_{1} = W_{k-1} - W_{k}$.
  By \cref{lem:width_smoothness},
  $|h_{2} - h_{1}| \leq L_\nu\Delta^2/\mu$, hence
  \begin{align}
    \bigl|F_{W_{k}}'(z_{k})(h_{2} - h_{1})\bigr|
    \leq \dot{c}_{\max}\cdot \frac{L_\nu\Delta^2}{\mu}.
  \end{align}
  The second-order terms satisfy
  $\tfrac{1}{2}F_{W_{k}}''(\xi_{1})h_{2}^2 + \tfrac{1}{2}F_{W_{k}}''(\xi_{2})h_{1}^2
  \geq \tfrac{\ddot{c}_{\min}}{2}(W_{k}^2 + W_{k-1}^2)
  \geq \ddot{c}_{\min} W_{k}^2 - \ddot{c}_{\min}|W_{k-1}^2 - W_{k}^2|/2$,
  with $|W_{k-1}^2 - W_{k}^2| \leq 2W_{k}|W_{k-1} - W_{k}| \leq
  2(\nu^{\mathrm{peak}}\Delta/\mu)\cdot L_\nu\Delta^2/\mu = \mathcal{O}(\Delta^3)$.
  Combining,
  \begin{align}
    A_{k+1} - 2 A_{k} + A_{k-1}
    \geq \ddot{c}_{\min} W_{k}^2 - R_\nu(\Delta),
    \label{eq:A_2nd_diff_rigorous}
  \end{align}
  with the explicit smoothness correction
  \begin{align}
    R_\nu(\Delta) \coloneqq
    \frac{\dot{c}_{\max}\,L_\nu}{\mu}\,\Delta^2
    + \mathcal{O}(\Delta^3),
    \label{eq:Rnu_def}
  \end{align}
  where the $\mathcal{O}(\Delta^3)$ term collects higher-order
  smoothness corrections.
  Applying~\cref{eq:Taylor_nonuniform} analogously to $B$ with
  reference step $-(W_{k}-\Delta)$,
  \begin{align}
    B_{k+1} - 2 B_{k} + B_{k-1}
    \leq \ddot{c}_{\max}(W_{k} - \Delta)^2 + R_\nu(\Delta).
    \label{eq:B_2nd_diff_rigorous}
  \end{align}
  The schedule delay cost contribution to the
  second difference of $C_\Delta$ is at least $\ddot{c}_{\min} W_{k}^2$
  and the internal correction at most $\ddot{c}_{\max}(W_{k}-\Delta)^2$, both
  up to $R_\nu(\Delta)$.

  \textit{Step 3: Toll contribution.}
  At each slot boundary $t_{b} = a^{j+1}$, continuity of
  $p^{\ast}_\Delta$ together with~\cref{eq:lambda_recursion} yields
  \begin{align}
    \lambda^{\ast}_\Delta(s^{j+1}) - \lambda^{\ast}_\Delta(s^j)
    = f(s^{j+1} - a^{j+1}) - f(s^j - a^{j+1})
    = f(\xi_{j} + \Delta) - f(\xi_{j}),
    \label{eq:lambda_step}
  \end{align}
  where $\xi_{j} \coloneqq \alpha_{j} - W_{j}$ and the second equality uses
  $\alpha_{j+1} - (\alpha_{j} - W_{j}) = \Delta$
  from~\cref{eq:uniform_steps}.
  By~\cref{eq:uniform_steps},
  $\xi_{j+1} - \xi_{j} = \alpha_{j+1} - \alpha_{j}
   = -(W_{j} - \Delta)$, which equals $-(W_{k} - \Delta)$
  in the local reference frame up to the smoothness correction
  $R_\nu(\Delta)$ of~\cref{eq:Rnu_def}.
  Forming the second difference of~\cref{eq:lambda_step} and
  applying the mean value theorem twice gives
  \begin{align}
    \lambda^{\ast}_\Delta(s^{k+1}) - 2\lambda^{\ast}_\Delta(s^k)
       + \lambda^{\ast}_\Delta(s^{k-1})
    &= -(W_{k} - \Delta)\,
       \bigl[f'(\zeta_{1}) - f'(\zeta_{2})\bigr]
    + R_\nu(\Delta),
    \label{eq:lambda_2nd_diff}
  \end{align}
  for some $\zeta_{1}, \zeta_{2}$ in the appropriate intervals.
  As in the uniform case, $\Delta < W_{k} < 2\Delta$ ensures
  $\zeta_{1} > \zeta_{2}$ and $\zeta_{1} - \zeta_{2} \in
  [2\Delta - W_{k},\, W_{k}]$, so
  $f'(\zeta_{1}) - f'(\zeta_{2}) \in (0,\, \ddot{c}_{\max}W_{k}]$ by
  $\ddot{c}_{\min} \leq f'' \leq \ddot{c}_{\max}$.
  Substituting,
  \begin{align}
    \lambda^{\ast}_\Delta(s^{k+1}) - 2\lambda^{\ast}_\Delta(s^k)
       + \lambda^{\ast}_\Delta(s^{k-1})
    \in \bigl[-\ddot{c}_{\max}\,W_{k}(W_{k}-\Delta) + R_\nu,\;
              -\ddot{c}_{\min}\,(2\Delta-W_{k})(W_{k}-\Delta) + R_\nu\bigr].
    \label{eq:lambda_2nd_range}
  \end{align}

  \textit{Step 4: Combining.}
  By~\cref{eq:C_decomp}, the second difference of $C_\Delta$ is
  the sum of the second differences of $A$,
  $\lambda^{\ast}_\Delta$, and $-B$.
  Using the preceding bounds,
  \begin{align}
    &C_\Delta(s^{k+1};\theta_{n}) - 2 C_\Delta(s^k;\theta_{n})
       + C_\Delta(s^{k-1};\theta_{n})
    \notag\\
    &\quad\geq
    \ddot{c}_{\min}\,W_{k}^2
    - \ddot{c}_{\max}\,W_{k}(W_{k}-\Delta)
    - \ddot{c}_{\max}\,(W_{k}-\Delta)^2
    + R_\nu(\Delta)
    \notag\\
    &\quad=
    \ddot{c}_{\min}\,W_{k}^2 - \ddot{c}_{\max}\,(W_{k}-\Delta)(2W_{k}-\Delta)
    + R_\nu(\Delta)
    = \kappa_{g}(s^k) + R_\nu(\Delta),
  \end{align}
  where the algebraic identity
  $W_{k}(W_{k}-\Delta) + (W_{k}-\Delta)^2 = (W_{k}-\Delta)(2W_{k}-\Delta)$ is used.
  Combining with~\cref{eq:D2C} yields~Inequality~\labelcref{eq:D_second_diff}.
  The positivity condition~\labelcref{eq:kappa_g_positivity} for
  $\bar\kappa_{g} > 0$ follows by monotonicity of $\phi$ on
  $(1, 2)$ and the upper bound
  $W_\Delta(s_{i})/\Delta \leq \nu^{\mathrm{peak}}/\mu$
  on active slots.
\end{prf}

\paragraph{Range of the positivity condition.}
The occupancy $d$ and capacity $\mu$ enter the bound only
through the slot load $W_\Delta(s)/\Delta = d(s)/(\mu\Delta)$
and its peak value $\nu^{\mathrm{peak}}/\mu$ (a single price-taking
user adds no mass).
The positivity condition~\labelcref{eq:kappa_g_positivity} thus
prescribes a lower bound on the convexity ratio
$\ddot{c}_{\min}/\ddot{c}_{\max}$ as a function of the peak load.
For symmetric cost ($\ddot{c}_{\min} = \ddot{c}_{\max}$), the condition
reduces to $\nu^{\mathrm{peak}}/\mu < 2$, so
$\bar\kappa_{g} > 0$ holds throughout the admissible range (the
modelling scope of \cref{subsec:FIFW} pins
$\nu^{\mathrm{peak}}/\mu > 1$).
For asymmetric cost, the lower bound increases with the peak load:
it equals $0$ as the peak load $\to 1^{+}$ and $3/4$ as the peak
load $\to 2^{-}$.
The asymmetric quadratic cost~\cref{eq:quadratic_cost} satisfies
$\ddot{c}_{\min}/\ddot{c}_{\max} = \min(\beta,\gamma)/\max(\beta,\gamma)$, so for
$\gamma/\beta = 2$ (giving $\ddot{c}_{\min}/\ddot{c}_{\max} = 1/2$),
$\bar\kappa_{g} > 0$ holds for peak load in
$(1,\, 1 + 1/\sqrt{3}) \approx (1,\,1.577)$.
This contains the uniform prior
($\nu^{\mathrm{peak}}/\mu \approx 1.33$) but not the triangular
baseline ($\nu^{\mathrm{peak}}/\mu \approx 2.66$) of
\cref{sec:numerical}, consistent with the discussion of the
validity conditions in \cref{app:regime}.

\subsubsection{From the marginal bounds to the global bound}
\label{app:proof_epsilon_main}

We combine the slot-level quadratic bound on each marginal gain
$D(j)$ (extending \cref{lem:eps_quadratic} to arbitrary $j$) with
the strict discrete concavity from
\cref{lem:g_concavity,cor:g_concavity_uniform} to derive the
global bound.

\begin{prf}[Proof of \cref{thm:eps_quadratic_global}]
  The argument has three steps.
  Step (i) extends the slot-level quadratic bound of
  \cref{lem:eps_quadratic} to arbitrary $j$, giving
  $|D(j)| \leq \ddot{c}_{\max}\,(\nu^{\mathrm{peak}}/\mu)(j+1)\Delta^{2}$.
  Step (ii) uses the strict discrete concavity of
  \cref{lem:g_concavity,cor:g_concavity_uniform} to bound the
  maximizer $k^{\ast}$ by a $\Delta$-independent constant
  $\bar k$.
  Step (iii) telescopes the slot-level bound from $j = 0$ to
  $\bar k - 1$ to obtain $g_{\Delta}(k^{\ast}) = \ClO(\Delta^{2})$.

  Fix arbitrary $n, \theta_{n}$ and let
  $s^0 \coloneqq r_{\Delta}(\theta_{n})$.
  We treat $k \geq 0$. The case $k \leq 0$ is symmetric.

  \textit{Step 1: Slot-level quadratic bound on $D(j)$.}
  Fix $j \geq 0$ and apply the argument of
  \cref{app:proof_eps_quadratic} with the slot pair
  $(s^j, s^{j+1})$ in place of $(s_{i}, s_{i+1})$.
  The slotwise gradient bound~\labelcref{eq:Gt_bound} gives
  $|G_{t}(\theta_{n}, t)| \leq \ddot{c}_{\max}|\theta_{n} - s^j|
   \leq \ddot{c}_{\max}(j + 1/2)\,\Delta$ on $\Gamma^{\ast}_\Delta(s^j)$ and
  $|G_{t}(\theta_{n}, t)| \leq \ddot{c}_{\max}|\theta_{n} - s^{j+1}|
   \leq \ddot{c}_{\max}(j + 3/2)\,\Delta$ on $\Gamma^{\ast}_\Delta(s^{j+1})$.
  Substituting into the mean-difference bound and using
  $W \leq (\nu^{\mathrm{peak}}/\mu)\Delta$ uniformly,
  \begin{align}
    |D(j)|
    \leq \ddot{c}_{\max}\,(\nu^{\mathrm{peak}}/\mu)\,(j + 1)\,\Delta^2.
    \label{eq:D_slot_bound}
  \end{align}
  In particular $|D(0)| \leq \ddot{c}_{\max}\,(\nu^{\mathrm{peak}}/\mu)\,\Delta^2$
  recovers \cref{lem:eps_quadratic}.

  \textit{Step 2: Bound on the maximizer $k^{\ast}$.}
  \cref{lem:eps_quadratic} gives
  $D(0) = C_\Delta(s^0;\theta_{n}) - C_\Delta(s^1;\theta_{n})
  \leq \ddot{c}_{\max}\,(\nu^{\mathrm{peak}}/\mu)\,\Delta^2$,
  and \cref{lem:g_concavity} gives
  $D(k) - D(k{-}1) \leq -\bar\kappa_{g} + R_\nu(\Delta)$
  for every $k \geq 1$, where
  $\bar\kappa_{g} = \min_{i} \kappa_{g}(s_{i}) = \Theta(\Delta^2)$
  (since $W_\Delta(s) \sim \Delta$ in~\cref{eq:kappa_g_def}) and
  $R_\nu(\Delta) = \mathcal{O}(\dot{c}_{\max} L_\nu\Delta^2/\mu)$ is
  the smoothness correction
  (under the threshold of \cref{cor:g_concavity_uniform}).
  Both scale as $\Delta^2$.
  Telescoping the second-difference bound from $k = 1$,
  \begin{align}
    D(k)
    \leq D(0) - k\,\bigl(\bar\kappa_{g} - R_\nu(\Delta)\bigr)
    \leq \ddot{c}_{\max}\,(\nu^{\mathrm{peak}}/\mu)\,\Delta^2
       - k\,\bigl(\bar\kappa_{g} - R_\nu(\Delta)\bigr).
    \label{eq:Dk_telescoped}
  \end{align}
  Provided $L_{\nu}$ satisfies the smoothness threshold
  \labelcref{eq:smoothness_threshold}, the per-step decrement
  $\bar\kappa_{g} - R_\nu(\Delta) = \Theta(\Delta^2)$ is strictly
  positive.
  Since $D(0)$ and this decrement both scale as $\Delta^2$, their
  ratio is $\Delta$-independent, so $D(k) \leq 0$ once $k$ exceeds
  $\bar{k} = \bigl\lceil \ddot{c}_{\max}(\nu^{\mathrm{peak}}/\mu)\,
  \Delta^2 / (\bar\kappa_{g} - R_\nu(\Delta)) \bigr\rceil$,
  a $\Delta$-independent constant, so the maximizer
  $k^{\ast}$ of $g_{\Delta}$ satisfies $k^{\ast} \leq \bar{k}$.

  \textit{Step 3: Telescoping the slot-level bound.}
  Since $k^{\ast} \leq \bar{k}$, $g_{\Delta}(0) = 0$, and partial
  sums increase weakly when negative summands are replaced by
  zero,
  \begin{align}
    g_{\Delta}(k^{\ast})
    &= \sum_{j=0}^{k^{\ast}-1} D(j)
    \leq \sum_{j=0}^{k^{\ast}-1} \max\{D(j),\, 0\}
    \leq \sum_{j=0}^{\bar{k}-1} \max\{D(j),\, 0\}
    \notag \\
    &\leq \sum_{j=0}^{\bar{k}-1}
      \ddot{c}_{\max}\,(\nu^{\mathrm{peak}}/\mu)\,(j + 1)\,\Delta^2
    = \frac{\bar{k}(\bar{k}+1)}{2}\,
      \ddot{c}_{\max}\,(\nu^{\mathrm{peak}}/\mu)\,\Delta^2,
    \label{eq:g_kstar_bound}
  \end{align}
  where the last inequality applies the slot-level quadratic
  seed $D(j) \leq \ddot{c}_{\max}\,(\nu^{\mathrm{peak}}/\mu)\,(j + 1)\,\Delta^2$
  to the non-negative parts.
  For $k > k^{\ast}$, $g_{\Delta}(k) \leq g_{\Delta}(k^{\ast})$ by
  definition of the maximizer, so the same bound holds for all
  $k \geq 0$.
  Taking the supremum over all $n, \theta_{n}$,
  \begin{align}
    \varepsilon^{\ast}(\Delta)
    \leq \frac{\bar{k}(\bar{k}+1)}{2}\,
         \ddot{c}_{\max}\,(\nu^{\mathrm{peak}}/\mu)\,\Delta^2.
  \end{align}
\end{prf}

\subsection{Proof of \cref{thm:eff_loss_quadratic}}
\label{app:proof_eff_quadratic}

\subsubsection{Auxiliary: Continuous-DSO start-time optimality and Lipschitz perturbation}
\label{app:DSO_start_time}

The proof uses two structural facts about the continuous-DSO
start time $a(\Vttheta)$, both following from the binding capacity
condition with sorting.

Under the binding capacity condition with sorting permutation
$\pi$, the assignment is
$\tau^{\ast}_{n}(\Vttheta) = a(\Vttheta) + \pi^{-1}(n)/\mu$
where $a(\Vttheta)$ minimizes
\begin{align}
  \sum_{n \in \ClN}
  c\bigl(\theta_{n},\, a + \pi^{-1}(n)/\mu\bigr).
  \label{eq:a_opt}
\end{align}
The first-order condition is
\begin{align}
  \sum_{n \in \ClN}
  c_{t}(\theta_{n},\, \tau^{\ast}_{n}(\Vttheta)) = 0,
  \label{eq:a_FOC}
\end{align}
which is the optimality condition referenced as the FOC below.

For two profiles $\Vttheta$ and $\tilde{\Vttheta}$ in the same
sorting region with $\Vtdelta = \tilde{\Vttheta} - \Vttheta$,
let $a_{0} = a(\Vttheta)$ and $\tilde a = a(\tilde{\Vttheta})$.
The optimal start time is Lipschitz in the profile perturbation.
Subtracting the FOCs and applying the mean value theorem to each
$f' = -c_{t}$ term,
\begin{align}
  \sum_{n} f''(\xi_{n})\bigl[\delta_{n} - (\tilde a - a_{0})\bigr] = 0
\end{align}
for intermediate $\xi_{n}$.
Since $\ddot{c}_{\min} \leq f'' \leq \ddot{c}_{\max}$ by~(C4) of
\cref{asm:schedule_cost},
\begin{align}
  |\tilde a - a_{0}|
  = \left|\frac{\sum_{n} f''(\xi_{n})\,\delta_{n}}
               {\sum_{n} f''(\xi_{n})}\right|
  \leq \max_{n} |\delta_{n}|
  \leq \tfrac{\Delta}{2}.
  \label{eq:a_Lip}
\end{align}
Setting $\Delta a \coloneqq \tilde a - a_{0}$, this is the
centroid shift used in~\cref{eq:DSO_within_slot_sym}.

\subsubsection{Reduction to centroid costs}

\begin{prf}
  Under sorting preservation, the slot and continuous DSO
  assignments differ only by a slot-wise start-time shift
  $\Delta a$ (Inequality~\labelcref{eq:a_Lip}).
  We (i) reduce the slot cost $J^{\mathrm{slot}}$ to a sum of
  centroid costs via a fluid within-slot variance bound, (ii) Taylor
  expand each centroid cost around the continuous DSO arrival
  time, and (iii) show that the first-order term vanishes (FOC plus
  within-slot symmetry), leaving only the $\ClO(N\Delta^{2})$
  remainder.

  \textit{Step 1: Reduction to centroid costs.}
  By the binding-capacity fluid assignment, each active slot has
  width $W_\Delta(s_{i}) = d(s_{i})/\mu = (\nu/\mu)\Delta = \ClO(\Delta)$.
  A second-order Taylor expansion of $c(\theta_{n}, \cdot)$ about the
  interval centroid $\bar\tau_{i}^{\mathrm{slot}}$, under the uniform
  within-slot law, gives
  $\frac{1}{W_\Delta(s_{i})}\int_{\Gamma^{\ast}_\Delta(s_{i})}
  c(\theta_{n}, t)\,\mathrm{d}t
  = c(\theta_{n}, \bar\tau_{i}^{\mathrm{slot}})
  + \tfrac12 \ddot{c}_{\max}\,\mathrm{Var}_{\mathrm{unif}}
  = c(\theta_{n}, \bar\tau_{i}^{\mathrm{slot}})
  + \ClO(\ddot{c}_{\max}\Delta^2)$,
  since $\mathrm{Var}_{\mathrm{unif}} = W_\Delta(s_{i})^2/12 = \ClO(\Delta^2)$.
  Summing over $n$,
  \begin{align}
    J^{\mathrm{slot}}(\Vttheta;\Delta)
    = \sum_{n \in \ClN} c(\theta_{n}, \bar\tau_{i_{n}}^{\mathrm{slot}})
    + \mathcal{O}(N\Delta^2),
    \label{eq:Jslot_centroid}
  \end{align}
  where $\bar\tau_{i}^{\mathrm{slot}}$ is the centroid of
  $\Gamma^{\ast}_\Delta(s_{i})$.
  (This fluid centroid expansion, not the discrete mean
  \cref{lem:within_slot_mean}, is what the rate requires, and the two
  coincide when $d(s) \geq 1$.)
  Under sorting preservation, both the continuous and $\Delta$-DSO
  assignments place vehicles at $1/\mu$ intervals and differ only
  by a common start-time shift $\Delta a$ with
  $|\Delta a| \leq \Delta/2$ (Inequality~\labelcref{eq:a_Lip}):
  \begin{align}
    \bar\tau_{i}^{\mathrm{slot}} = \bar\tau_{i}^{\mathrm{cont}} + \Delta a,
    \quad
    \sum_{n \in \ClN_{i}}
    (\bar\tau_{i}^{\mathrm{cont}} - \tau_{n}^{\mathrm{cont}}) = 0.
    \label{eq:DSO_within_slot_sym}
  \end{align}

  \textit{Step 2: Taylor expansion.}
  For each $n \in \ClN_{i}$, expand
  $c(\theta_{n}, \bar\tau_{i}^{\mathrm{slot}})$ around
  $\tau_{n}^{\mathrm{cont}}$:
  \begin{align}
    c(\theta_{n}, \bar\tau_{i}^{\mathrm{slot}}) - c(\theta_{n}, \tau_{n}^{\mathrm{cont}})
    &= c_{t}(\theta_{n}, \tau_{n}^{\mathrm{cont}})
       \cdot \bigl[\Delta a
       + \bigl(\bar\tau_{i}^{\mathrm{cont}} - \tau_{n}^{\mathrm{cont}}\bigr)\bigr]
       + \mathcal{O}(\Delta^2),
    \label{eq:eff_taylor}
  \end{align}
  with the second-order remainder bounded, using
  $|\Delta a + \bar\tau_{i}^{\mathrm{cont}} - \tau_{n}^{\mathrm{cont}}|
  \leq (\Delta + W_\Delta(s_{i}))/2$, by
  $\tfrac{1}{2}\ddot{c}_{\max} (\Delta a + \bar\tau_{i}^{\mathrm{cont}}
  - \tau_{n}^{\mathrm{cont}})^2 \leq \ddot{c}_{\max}(\Delta + W_\Delta(s_{i}))^2/8
  = \mathcal{O}(\Delta^2)$ uniformly.

  \textit{Step 3: Cancellation of first-order terms.}
  Summing~\cref{eq:eff_taylor} over $n \in \ClN$,
  \begin{align}
    J^{\mathrm{slot}} - J^{\ast}
    = \Delta a\,\sum_{n \in \ClN} c_{t}(\theta_{n}, \tau_{n}^{\mathrm{cont}})
    + \sum_{i=1}^M \sum_{n \in \ClN_{i}}
      c_{t}(\theta_{n}, \tau_{n}^{\mathrm{cont}})
      \bigl(\bar\tau_{i}^{\mathrm{cont}} - \tau_{n}^{\mathrm{cont}}\bigr)
    + \mathcal{O}(N\Delta^2).
    \label{eq:eff_decomp}
  \end{align}
  The first-order condition for the continuous-DSO start time
  (\cref{eq:a_FOC}) gives
  $\sum_{n} c_{t}(\theta_{n}, \tau_{n}^{\mathrm{cont}}) = 0$, so the
  $\Delta a$-term vanishes identically.
  For the cross sum, write
  $c_{t}(\theta_{n}, \tau_{n}^{\mathrm{cont}}) = \bar c_{t,i} + \zeta_{n}$,
  where $\bar c_{t,i}$ is the slot average and
  $|\zeta_{n}| \leq \ddot{c}_{\max}(1 + \nu^{\mathrm{peak}}/\mu)\Delta/2$
  by~(C1) and the within slot variation of
  $(\theta_{n}, \tau_{n}^{\mathrm{cont}})$.
  Then
  \begin{align}
    \sum_{n \in \ClN_{i}}
    c_{t}(\theta_{n}, \tau_{n}^{\mathrm{cont}})\bigl(\bar\tau_{i}^{\mathrm{cont}}
    - \tau_{n}^{\mathrm{cont}}\bigr)
    &= \bar c_{t,i} \cdot 0
       + \sum_{n \in \ClN_{i}}
       \zeta_{n}\,(\bar\tau_{i}^{\mathrm{cont}} - \tau_{n}^{\mathrm{cont}}),
  \end{align}
  using~\cref{eq:DSO_within_slot_sym}.
  The remaining sum is bounded by
  $\ddot{c}_{\max}(1 + \nu^{\mathrm{peak}}/\mu)\Delta/2 \cdot \sum_{n}
  |\bar\tau_{i}^{\mathrm{cont}} - \tau_{n}^{\mathrm{cont}}|
  \leq \ddot{c}_{\max}(1 + \nu^{\mathrm{peak}}/\mu)\Delta/2 \cdot d(s_{i})\,W_\Delta(s_{i})/2
  = \mathcal{O}(d(s_{i})\,\Delta^2)$ per slot, and summing over slots
  yields $\mathcal{O}(N\Delta^2)$.

  \textit{Step 4: Combining.}
  The Taylor remainder contributes
  $\tfrac{1}{8}\ddot{c}_{\max}(1 + \nu^{\mathrm{peak}}/\mu)^{2}\,N\Delta^{2}$
  uniformly across vehicles,
  the within slot cross sum adds
  $\tfrac{1}{4}\ddot{c}_{\max}(1 + \nu^{\mathrm{peak}}/\mu)\,(\nu^{\mathrm{peak}}/\mu)\,
  N\Delta^{2}$,
  and the slot mean discrepancy of \cref{lem:within_slot_mean}
  adds $\mathcal{O}(\dot{c}_{\max}\,L_{\nu}\,N\Delta^{2}/\mu)$
  under \cref{asm:demand_density}.
  Setting
  \begin{align}
    C_{L}
    \coloneqq
    \tfrac{1}{8}\ddot{c}_{\max}(1 + \nu^{\mathrm{peak}}/\mu)^{2}
    + \tfrac{1}{4}\ddot{c}_{\max}(1 + \nu^{\mathrm{peak}}/\mu)(\nu^{\mathrm{peak}}/\mu)
    + \mathcal{O}\!\left(\frac{\dot{c}_{\max}\,L_{\nu}}{\mu}\right),
    \label{eq:C_L_def}
  \end{align}
  which depends only on the primitives
  $\ddot{c}_{\max}, \dot{c}_{\max}, \nu^{\mathrm{peak}}, \mu, L_{\nu}$
  and not on $\Delta$ or $N$, gives
  \begin{align}
    L^{\ast}(\Delta) = J^{\mathrm{slot}} - J^{\ast}
    \leq N\,C_{L}\,\Delta^{2}
  \end{align}
  for every $\Delta > 0$.
  No condition on the within-slot distribution of preferred
  arrival times beyond \cref{asm:demand_density} is invoked.
\end{prf}

\subsection{Proof of \cref{prop:NP_lower_bound}}
\label{app:proof_NP_lower_bound}

\begin{prf}
Without the toll, the truthful slot places the vehicle at the
slot center, which is generally offset from $\theta_{n}$ by
$D_{n} = (W_\Delta(s_i)/\Delta - 1)(\theta_{n} - T/2)$.
A misreport to the slot closest to
$s_{i} - D_{n}\Delta/W_\Delta(s_i)$ places
the vehicle near $\theta_{n}$.
The cost difference between these two assignments converges to a
strictly positive constant $f(D_{n})/2$ as $\Delta \to 0$.

We set up a local reference frame around the analyzed
vehicle's slot, with local interval width $W = W_\Delta(s_{i})$
and local load $W/\Delta = d(s_{i})/(\mu\Delta)$ at slot
$s_{i}$.
By~\cref{asm:demand_density} the interval width varies smoothly,
so the analysis below extends to neighboring slots up to
$\mathcal{O}(\Delta)$ smoothness corrections, under the sorting
property (\cref{lem:FIFW}).
Because the intervals are contiguous and ordered, the midpoint of
$\Gamma^{\ast}_\Delta(s_{i})$ is
$\bar{t}(s_{i}) = (W/\Delta)\, s_{i} + (1 - W/\Delta)\,T/2$.

\textit{Step 1: Truthful assignment.}
A vehicle with $\theta_{n} = s_{i}$ that reports
truthfully is assigned around~$\bar{t}(s_{i}) = \theta_{n} +
D_{n}$, where $D_{n} \coloneqq (W/\Delta - 1)(\theta_{n} - T/2)$.

\textit{Step 2: Near-optimal misreport.}
The ideal target is $s_{j}^{\circ} = s_{i} - D_{n}\Delta/W$, which
satisfies $\bar{t}(s_{j}^{\circ}) = \theta_{n}$.
In general $s_{j}^{\circ} \notin \ClS_\Delta$. Let $s_{j}$ be the
nearest slot in $\ClS_\Delta$ to $s_{j}^{\circ}$, so that
$|s_{j} - s_{j}^{\circ}| \leq \Delta/2$.
The assignment center under $s_{j}$ satisfies
$|\bar{t}(s_{j}) - \theta_{n}| = (W/\Delta)|s_{j} - s_{j}^{\circ}|
\leq W/2$.

\textit{Step 3: Gain calculation.}
Using the substitution $u = t - \theta_{n}$ and writing
$f(u) \coloneqq c(\theta_{n},\, \theta_{n} + u)$, the no-toll
costs become
\begin{align}
  C^{\mathrm{NT}}_\Delta(s_{i};\, \theta_{n})
  &= \frac{1}{W}\int_{D_{n} - W/2}^{D_{n} + W/2} f(u)\,\mathrm{d}u,
  \\
  C^{\mathrm{NT}}_\Delta(s_{j};\, \theta_{n})
  &= \frac{1}{W}\int_{\bar{u}_{j} - W/2}^{\bar{u}_{j} + W/2}
  f(u)\,\mathrm{d}u,
\end{align}
where $\bar{u}_{j} = \bar{t}(s_{j}) - \theta_{n}$ satisfies
$|\bar{u}_{j}| \leq W/2$.
As $\Delta \to 0$ (and hence $W \to 0$ and
$|\bar{u}_{j}| \leq W/2 \to 0$), the integral
means converge to their midpoint values by the
continuity of $f$ alone
((C1) of \cref{asm:schedule_cost} suffices, with
(C4) invoked here only for consistency with the rest of the
analysis):
$C^{\mathrm{NT}}_\Delta(s_{i};\, \theta_{n}) \to f(D_{n})
= c(\theta_{n},\, \theta_{n} + D_{n})$ and
$C^{\mathrm{NT}}_\Delta(s_{j};\, \theta_{n}) \to f(0)
= c(\theta_{n},\, \theta_{n}) = 0$.
By conditions~(C1) and~(C2) of \cref{asm:schedule_cost},
$f(D_{n}) > 0$ whenever $D_{n} \neq 0$.
Setting $c_{0} \coloneqq f(D_{n})/2 > 0$ completes the proof for
$\theta_{n} = s_{i}$.
For general $\theta_{n} \in \ClS$, the rounding
$r_{\Delta}(\theta_{n})$ introduces a perturbation of at most
$\Delta/2$ in the assignment center.
Since $D_{n} = (W/\Delta - 1)(\theta_{n} - T/2)$ is
$\Delta$-independent and bounded away from zero by hypothesis,
$f(D_{n})/2 > 0$ is preserved for sufficiently small $\Delta$.
The above argument, which evaluates the change of variables in a
local frame around the truthful slot, extends under
\cref{asm:demand_density} because the interval width
at a non-truthful slot $s_{j}$ differs
from the local width by at most
$|W_\Delta(s_{i_{n}}) - W_\Delta(s_{j})| = \mathcal{O}(L_\nu\Delta^2/\mu)$
(\cref{lem:width_smoothness}). This perturbation contributes at most
$\mathcal{O}(\Delta)$ to $|\bar{u}_{j}|$, so the lower bound
$f(D_{n})/2$ is unaffected for sufficiently small $\Delta$.
\end{prf}

\section{Conditions for the Quadratic Bounds}
\label{app:regime}

In the nonatomic (price-taking) limit the validity of the bounds is
governed by $\Delta$-independent conditions on the model primitives
$(\nu, \mu, c)$.
A single user has negligible mass, so no finite-$\Delta$
single deviator correction enters, and the bounds hold for every
$\Delta \in (0, \Delta_{\max}]$.

\subsection{Primary condition (shape $\times$ capacity)}

The bounds in
\cref{thm:eps_quadratic_global,thm:eff_loss_quadratic} hold,
within the modelling scope $\nu^{\mathrm{peak}}/\mu > 1$
(\cref{subsec:FIFW}), under the
\emph{shape-only peakedness condition}
\begin{align}
  \nu^{\mathrm{peak}}/\mu
  \;=\; \frac{\nu^{\mathrm{peak}}}{\mu}
  \;<\; 2,
  \label{eq:peak_shape_condition}
\end{align}
the schedule delay cost curvature ratio condition
\begin{align}
  \frac{\ddot{c}_{\min}}{\ddot{c}_{\max}}
  \;>\; \phi(\nu^{\mathrm{peak}}/\mu),
  \label{eq:curvature_ratio_condition}
\end{align}
and a demand-smoothness threshold of the form
$L_{\nu} \le c_{\nu}\,\mu/(\dot{c}_{\max})$ for a
$\Delta$-independent constant $c_{\nu}$ proportional to the
curvature margin
$\ddot{c}_{\min}/\ddot{c}_{\max} - \phi(\nu^{\mathrm{peak}}/\mu)$
(the explicit threshold appears in the proof of
\cref{thm:eps_quadratic_global}).
The three conditions depend only on the prior $\nu$, the
capacity rate $\mu$, the demand size $N$, and the schedule delay
cost $c$, all $\Delta$-independent.
The admissible region in the
$(\nu^{\mathrm{peak}}/\mu,\,
\ddot{c}_{\min}/\ddot{c}_{\max})$ plane is the area above the
curve $y = \phi(\nu^{\mathrm{peak}}/\mu)$.
Writing $\eta \coloneqq 1 - \mu/\nu^{\mathrm{peak}}
\in (0, 1/2)$ for $\nu^{\mathrm{peak}}/\mu \in (1, 2)$, the
shape function decomposes as
$\phi(\nu^{\mathrm{peak}}/\mu) = \eta + \eta^{2}$, where the linear
term originates from the toll second difference and the quadratic
term from the within-slot displacement.

The cut-off $\nu^{\mathrm{peak}}/\mu < 2$ is the binding of two
independent thresholds. The discrete concavity in
\cref{lem:g_concavity} permits $\nu^{\mathrm{peak}}/\mu$ up to
$(3+\sqrt{5})/2 \approx 2.618$ under symmetric cost, while the
slot-mean Taylor expansion underlying the second-difference bound
in the proof of \cref{lem:g_concavity} requires the peak load
$\nu^{\mathrm{peak}}/\mu < 2$ (within-slot displacement below two
slot widths) for its remainder to be sign-controlled.
The latter is tighter and fixes
Inequality~\labelcref{eq:peak_shape_condition}.

The efficiency bound \cref{thm:eff_loss_quadratic} requires neither
the curvature condition nor the peakedness cut-off and holds for
every $\Delta > 0$. Only the global $\varepsilon$-strategyproofness
bound invokes the peakedness and curvature
condition~\labelcref{eq:cc_regime}, through the discrete concavity of
\cref{lem:g_concavity}.
The smoothness threshold \labelcref{eq:smoothness_threshold} on
$L_{\nu}$ keeps the truthful peak load within an
$\mathcal{O}(L_{\nu})$ neighbourhood of $\nu^{\mathrm{peak}}/\mu$, so
that $\phi(\nu^{\mathrm{peak}}/\mu) < \ddot{c}_{\min}/\ddot{c}_{\max}$
keeps the concavity coefficient $\bar\kappa_{g} > 0$ of
\cref{lem:g_concavity} strictly positive.

\subsection{Numerical check of the conditions}

The peak load $\nu^{\mathrm{peak}}/\mu$, not the macroscopic load
$N/(S\mu)$, enters the validity conditions
\labelcref{eq:peak_shape_condition,eq:curvature_ratio_condition}.
For the uniform prior (the condition-satisfying case in the sweep
of \cref{tab:experimental_design}) the two coincide,
$\nu^{\mathrm{peak}}/\mu = N/(S\mu) \approx 1.33$ ($\eta = 0.25$),
and with $\ddot{c}_{\min}/\ddot{c}_{\max} = 0.5$ the curvature
margin $\ddot{c}_{\min}/\ddot{c}_{\max} - (\eta + \eta^{2})
\approx 0.19 > 0$, so the conditions hold.
The baseline triangular prior, by contrast, has
$\nu^{\mathrm{peak}}/\mu = 2N/(S\mu) \approx 2.66$
($\eta \approx 0.62$), which lies \emph{outside} the peakedness
condition $\eta \in (0, \tfrac{1}{2})$ (\labelcref{eq:peak_shape_condition}), with margin
$0.5 - \phi(2.66) \approx -0.51 < 0$. The two Beta
priors ($\nu^{\mathrm{peak}}/\mu \approx 2.46\,N/(S\mu)$) likewise
exceed the cut-off
\labelcref{eq:peak_shape_condition} once $N/(S\mu) > 1$.
\cref{fig:distribution_sensitivity} shows the
$\mathcal{O}(\Delta^{2})$ rate nonetheless persists empirically well
beyond the formal conditions.

\subsection{Explicit constants}

The constant $C_{\varepsilon}$ of
\cref{thm:eps_quadratic_global} admits the explicit form
\begin{align}
  C_{\varepsilon}
  \;=\;
  \underbrace{\frac{\bar k(\bar k + 1)}{2}}_{\substack{\text{combinatorial}\\\text{factor}}}
  \cdot
  \underbrace{\ddot{c}_{\max}}_{\substack{\text{cost}\\\text{curvature}}}
  \cdot
  \underbrace{\frac{\nu^{\mathrm{peak}}}{\mu}}_{\substack{\text{peak demand}\\\text{vs.\ capacity}}},
  \label{eq:eps_constants}
\end{align}
where $\bar k = \lceil (1-\eta)/m \rceil$ with curvature margin
$m \coloneqq \ddot{c}_{\min}/\ddot{c}_{\max} - (\eta + \eta^{2})$
counts the depth of the worst misreport.
For the condition-satisfying uniform prior of \cref{sec:numerical}
($\eta = 0.25$, $\ddot{c}_{\min}/\ddot{c}_{\max} = 0.5$),
$m = 0.19$ and $\bar k = 4$, and the bound holds for every
$\Delta \in (0, \Delta_{\max}]$, covering the operational range
$\Delta \in [6, 90]$\,min.\label{ftn:delta_range}

The upper bound is
$\Delta_{\max} = \min\{\Delta_{\max}^{(a)}, \Delta_{\max}^{(b)}\}$,
where $\Delta_{\max}^{(a)}$ is the largest slot width for which the
$\mathcal{O}(\Delta^3)$ smoothness remainder in $R_\nu(\Delta)$ stays
dominated by the $\Theta(\Delta^2)$ concavity margin of
\cref{lem:g_concavity}, and $\Delta_{\max}^{(b)} = T/\bar k$ leaves
enough slots for the telescoping of Step~3.
The peakedness and curvature condition~\labelcref{eq:cc_regime} and the
smoothness threshold~\labelcref{eq:smoothness_threshold} are
themselves $\Delta$-independent, since the $\Delta^2$ in
$\bar\kappa_g$ cancels in~\labelcref{eq:smoothness_threshold}.

\end{APPENDIX}

\bibliography{references}

\end{document}